\documentclass[a4paper,11pt]{article}
\usepackage{amsmath, amssymb,amsthm,graphicx}
\usepackage{xcolor}
\usepackage[a4paper]{geometry}
\usepackage{framed}
\usepackage{array}
\usepackage{bm}
\usepackage{soul}

\newcommand{\scal}[2]{\langle #1, #2\rangle} 
\newcommand{\dd}{\mathrm{d}}
\newcommand{\NN}{\mathbb N}
\newcommand{\RR}{\mathbb R}
\newcommand{\CC}{\mathbb C}

\newtheorem{theorem}{Theorem}[section]
\newtheorem{lemma}[theorem]{Lemma}
\newtheorem{proposition}[theorem]{Proposition}
\newtheorem{corollary}[theorem]{Corollary}

\theoremstyle{definition}
\newtheorem{definition}{Definition}
\newtheorem{example}[theorem]{Example}
\newtheorem{remark}[theorem]{Remark}

\begin{document}
\title{The fundamental solution of a class of ultra-hyperbolic operators on Pseudo $H$-type groups}
\author{
Wolfram Bauer and Andr\'{e} Froehly   \thanks{All authors have been supported by the DAAD-NFR project  "Subriemannian structures on Lie groups, differential forms and PDE"; 
project number (NFR) 267630/F10 and (DAAD) 57344898.} \\
{\small Institut f\"{u}r Analysis, Leibniz Universit\"{a}t Hannover,} \\
{\small Welfengarten 1, 30167 Hannover, Germany}\\
{\small E-mail: \texttt{bauer@math.uni-hannover.de}}\\
{\small E-mail: \texttt{andre.froehly@math.uni-hannover.de}}\\
\vspace{1ex}\\ 
Irina Markina $ \mbox{}^*$\\
{\small Department of Mathematics, University of Bergen, },\\
{\small P.O. Box 7800, Bergen N-5020, Norway} \\
{\small E-mail: \texttt{Irina.Markina@math.uib.no}}}
\maketitle

\begin{abstract}
Pseudo $H$-type Lie groups $G_{r,s}$ of signature $(r,s)$ are defined via a module action of the Clifford algebra $C\ell_{r,s}$ on a vector space $V \cong \mathbb{R}^{2n}$. 
They form a subclass of all 2-step nilpotent Lie groups and based on their algebraic structure they can be equipped with 
a left-invariant pseudo-Riemannian metric. Let  $\mathcal{N}_{r,s}$ denote the Lie algebra corresponding to $G_{r,s}$. A choice of left-invariant vector fields 
$[X_1, \ldots, X_{2n}]$ which generate a complement of the center of $\mathcal{N}_{r,s}$  gives rise to a second order operator 
\begin{equation*}
\Delta_{r,s}:= \big{(}X_1^2+ \ldots + X_n^2\big{)}- \big{(}X_{n+1}^2+ \ldots + X_{2n}^2 \big{)}, 
\end{equation*}
which we call ultra-hyperbolic. In terms of classical special functions we present families of fundamental solutions of $\Delta_{r,s}$ in the case $r=0$, $s>0$ and study their properties. 
In the case of $r>0$ we prove that $\Delta_{r,s}$ admits no fundamental solution in the space of tempered distributions. Finally we discuss the local solvability of $\Delta_{r,s}$ and the 
existence of a fundamental solution in the space of Schwartz distributions. 
\\
{\bf keywords:} left invariant homogeneous differential operator, distributions, local solvability, Bessel functions
\\
{\bf MSC 2010:} primary: 65M80, secondary: 22E25
\end{abstract}
\section{Introduction}
\label{Introduction}
The study of existence and explicit representations of a fundamental solution to various geometrically induced differential operators has stimulated some research during the last decades (cf. 
\cite{BGP,Card_Saal,Gelfand_Shilov,KobayashiOrsted,MuellerRicci, MuellerRicci_2,DRham,Tie}). In this paper we are concerned with such a problem in case of a second order homogeneous 
differential operator $\Delta_{r,s}$ which is induced by a pseudo $H$-type Lie group $G_{r,s}$ (see Section \ref{Pseudo_H_type_groups} for the precise definitions). We write 
$\mathcal{N}_{r,s}$ for the Lie algebra of $G_{r,s}$ and in the standard way we identify $G_{r,s}$ and $\mathcal{N}_{r,s}$ through the exponential map. The Lie algebra 
$\mathcal{N}_{r,s}$ is nilpotent of step two and it admits a decomposition $\mathcal{N}_{r,s}=V\oplus Z$ into its center $Z$ and a complement. The subspace $V$ is even dimensional and generates 
a left invariant distribution on $G_{r,s}$ spanned by $\{X_1,\ldots,X_{2n}\}$. Moreover, $G_{r,s}$ is equipped with a pseudo-Riemannian left-invariant metric which has signature $(n,n)$ when it is 
restricted to the distribution. The differential operator $\Delta_{r,s}$ induced by this setup has the explicit form 
\begin{equation}\label{Definition_Ultra_hyperbolic_operator}
\Delta_{r,s}:= \big{(}X_1^2+ \ldots + X_n^2\big{)}- \big{(}X_{n+1}^2+ \ldots + X_{2n}^2 \big{)}. 
\end{equation}
Due to its similarity with the classical ultra-hyperbolic operator 
\begin{equation}\label{Definition_classical_UHO_introduction}
\mathcal{L}= \sum_{j=1}^n \frac{\partial^2}{\partial x_j^2}- \frac{\partial^2}{\partial x_{n+j}^2}
\end{equation}
 on $\mathbb{R}^{2n}$ we call $\Delta_{r,s}$ an ultra-hyperbolic operator on the Lie group $G_{r,s}$. We note that $\Delta_{r,s}$ degenerates on the center $Z$ and different from $\mathcal{L}$ it 
 has non-constant coefficients such that the theorem of Malgrange-Ehrenpreis on the existence of a fundamental solution is not at disposal. In the 
 special case where $G_{r,s}$ is the Heisenberg group $G_{0,1}$ an explicit representation of a fundamental solution of (\ref{Definition_Ultra_hyperbolic_operator}) has been previously 
 obtained in \cite{MuellerRicci_2,Tie}. 
 
 The global solutions to the equation $\Delta_{r,s}u=0$ and some of its generalizations on the Heisenberg group has been studied in~\cite{Kable1, Kable2}. 
 Therein the Heisenberg group was realized as a flag manifold $G/Q$ of some classical group $G$ factored by a parabolic subgroup $Q$ and the differential operators are acting on the 
 sections of certain line bundles. It has been observed in ~\cite{AltomaniSanti, FGMMV}  that some of the pseudo $H$-type groups can also be interpreted as flag manifolds. In these cases 
 the action of the ultra-hyperbolic operator may be considered on the corresponding bundle.  
 
The goals of this paper are as follows: 
 \begin{itemize}
 \item[(a)] Characterize the pairs $(r,s)$ for which the ultra-hyperbolic operator $\Delta_{r,s}$ admits a fundamental solution within the space of tempered or Schwartz distributions.  
 \item[(b)] Explicitly derive a class of fundamental solutions in the space of tempered distributions in all cases in which the existence is guaranteed. 
 \item[({c})] Characterize the local solvability of the ultra-hyperbolic operator $\Delta_{r,s}$. 
 \end{itemize}
\par
 Our strategy of deriving the fundamental solution of (\ref{Definition_Ultra_hyperbolic_operator}) in the most general setting is based on the following observation by J. Tie in \cite{Tie}. In the case of 
 signature $(r,s)=(0,1)$ it can be verified that a suitable change from real to complex variables in (\ref{Definition_Ultra_hyperbolic_operator}) transforms the ultra-hyperbolic operator $\Delta_{0,1}$ into a 
 sub-Laplace operator $\Delta_{\textup{sub}}$ on $G_{0,1}$. Based on the sub-ellipticity of $\Delta_{\textup{sub}}$ the existence of the heat kernel and a fundamental solution is guaranteed. Performing the same transformation to an explicit form of the fundamental solution of $\Delta_{\textup{sub}}$, produces -  after a regularization 
 procedure - a fundamental solution of the ultra-hyperbolic operator $\Delta_{0,1}$. In case of an arbitrary pseudo $H$-type group with signature $(r,s)$ a similar change of variables in $\Delta_{r,s}$ formally 
 gives the sub-Laplace operator  $\Delta_{\textup{sub}}$ corresponding to a certain 2-step nilpotent Lie group. 
 The heat kernel of $\Delta_{\textup{sub}}$ is known explicitly (c.f. \cite{BGG, OvFuIw,Furu}) and can be used to calculate a fundamental solution of $\Delta_{\textup{sub}}$ via an integration over the ''time variable''. 
 In the same way as before the corresponding change of variables in the fundamental solution of $\Delta_{\textup{sub}}$ produces a formal expression of a fundamental solution of $\Delta_{r,s}$. 
 However, a regularization process which rigorously defines a tempered distribution seems only possible in the case $r=0$. As it turns out this is not an artifact. In the last part of the paper we will show 
 that for $r>0$ a fundamental solution of $\Delta_{r,s}$ in the space of tempered distributions or even within the Schwartz distributions does not exist. 
 \vspace{1ex}\par 
 The non-uniqueness of a fundamental solution for a class of second order differential operators on a Lie group which contains $\Delta_{r,s}$ in a very special case was observed in \cite{MuellerRicci, MuellerRicci_2,Tie}. In our general setting we will present an uncountable family of fundamental solutions of  $\Delta_{0,s}$, $s \geq 1$  in (\ref{Definition_Ultra_hyperbolic_operator}) and relate them to classical 
 special functions (e.g. Bessel functions of the first and second kind). One of the fundamental solutions $K_{0,1}$ of the ultra-hyperbolic operator on the Heisenberg group  obtained in \cite{Tie} coincides with an iterated integral expression which previously  and by different methods was presented by D. M\"{u}ller and F. Ricci in \cite{MuellerRicci_2}.  
 This fact was already noticed by J. Tie in~\cite{Tie} and since our approach generalizes the expression presented there, it is not surprising that we can detect $K_{0,1}$ among the above mentioned family of distributions. In a second step we generalize D. M\"{u}ller and F. Ricci's formula to the case $s>1$ and $r=0$. 
 \vspace{1ex}\par
 Finally we consider the ultra-hyperbolic operator $\Delta_{r,s}$ in the case $r>0$. It is known that a left-invariant  differential operator on a Lie group $G$ in general does not possesses a
global fundamental solution among the tempered distributions $\mathcal{S}^{\prime}(G)$. 
Additional assumptions on the group and the operator have to be made to guarantee the existence (see \cite{Rais} and the references therein).   
 In Theorem \ref{theorem_existence_of_fundamental_solution_r>0} we prove that in the case $r >0$ the ultra-hyperbolic operator $\Delta_{r,s}$ does not admit a fundamental solution in 
 $\mathcal{S}^{\prime}(G_{r,s})$. One may pose the question whether one can invert $\Delta_{r,s}$ in the larger space $\mathcal{D}^{\prime}(G_{r,s})$ of Schwartz distributions. On the homogeneous 
 Lie group $G_{r,s}$ this question is known to be related to the {\it local solvability} of the operator (see \cite{Battesti,Mueller}). Based on a criterion by D. M\"{u}ller in \cite[Theorem \ref{theorem_Mueller_non_local_solvability}]{Mueller} we show that  $\Delta_{r,s}$ is {\it not locally solvable} if and only if $r>0$. Using this fact we can answer the above 
 question in a negative sense and prove non-existence of a fundamental solution of $\Delta_{r,s}$, $r>0$ in $\mathcal{D}^{\prime}(G_{r,s})$. 
 \vspace{1ex}\\ 
 The paper is organized as follows. \par 
In Section \ref{Pseudo_H_type_groups} we recall the notion of a pseudo $H$-type group attached to a $C\ell_{r,s}$-Clifford module and we present two low dimensional examples of such groups. In suitable coordinates we explicitly represent the left-invariant ultra-hyperbolic operator $\Delta_{r,s}$ as a difference of two {\it sum-of-squares operators}. 

Section \ref{Section_2} contains more details on the ultra-hyperbolic operator together with a few technical calculations that play a role in the subsequent analysis. In particular, we consider 
$\Delta_{r,s}$ after a partial Fourier transform as an operator on the Schwartz space. 

Via a suitable change from real to complex coordinates we relate the ultra-hyperbolic operator $\Delta_{r,s}$ to a (hypo-elliptic) sub-Laplace operator $\Delta_{\textup{sub}}$ on a 2-step nilpotent Lie group 
in Section \ref{Section_From_Sub-Laplacian_to_UHO}. 

In Section \ref{section_3} we rigorously show that the distribution derived in Section \ref{Section_From_Sub-Laplacian_to_UHO} in fact defines a fundamental solution $K_{0,s}$ of $\Delta_{0,s}$, i.e., 
in the case where $r=0$ and $s\geq 1$. 

We note that  $K_{0,s}$ is not the unique fundamental solution and we relate it to classical special function (Bessel functions of the 
first and second kind) in Section \ref{Section_non-uniqueness-of-the-fundamental-solution}. As a result we present an uncountable family of fundamental solutions. 

In Section \ref{Section_A second form of  fundamental solutions} we present a second form of a (specific)  fundamental solution $K_{0,s}$ of $\Delta_{0,s}$ which was obtained in the 
previous chapter. In the case of $G_{0,1}$ this form coincides with the expression obtained in \cite{MuellerRicci,Tie}. However, since the condition $s=1$ is not 
required in our setup we have obtained a generalization of the distributions in \cite{MuellerRicci,Tie} to Lie groups $G_{0,s}$ with center dimension $s>1$. 

Section \ref{Section_Singular_set_K_0_s_+} discusses the singularities of a fundamental solution in the case $r=0$. In particular, we determine a cone in $\mathbb{R}^{2n+s}$ containing its 
singular support. 

The remaining two chapters are concerned with the invertibility of $\Delta_{r,s}$ in the case $r>0$. In Section \ref{s_r_>_0} we prove that in this case no fundamental solution among the tempered distributions 
exists. In the final Section \ref{Section_Local_solvability} we relate our results to the problem of {\it local solvability} of $\Delta_{r,s}$ for $r>0$. Based on a theorem by D. M\"{u}ller in \cite{Mueller} 
we show non-local solvability and non-existence of a fundamental solution in the Schwartz distributions $\mathcal{D}^{\prime}(G_{r,s})$. 

In the appendix we relate a family of classical distributions associated to a non-degenerate quadratic form on $\mathbb{R}^{2n}$ in \cite{Gelfand_Shilov} to the representation of the 
fundamental solution of $\Delta_{0,s}$ which was obtained in Chapter \ref{Section_A second form of  fundamental solutions}. 
\section{Pseudo $H$-type groups}
\label{Pseudo_H_type_groups}
\label{section_introduction}
Let $r,s \in \mathbb{N}_0$ and consider $\mathbb{R}^{r,s}= \mathbb{R}^{r+s}$ with the non-degenerate bilinear form (scalar product)
\begin{equation*}
\scal{x}{y}_{r,s}= \sum_{i=1}^r x_iy_i- \sum_{j=1}^s x_{r+j}y_{r+j} 
\end{equation*}
and the corresponding quadratic form $q_{r,s}(x):= \langle x,x \rangle_{r,s}$. We denote by $C\ell_{r,s}$  the Clifford algebra generated by $(\RR^{r,s}, q_{r,s})$. Let $V$ be a 
$C\ell_{r,s}$-Clifford module, i.e., $V$ is a real vector space with module action 
\begin{equation*}
J: C\ell_{r,s} \times V \rightarrow V . 
\end{equation*}
For $z \in \RR^{r,s}$ we use the notation $J_z := J(z, \cdot) :V \rightarrow V$ . 
Moreover, we write $\mathbb{R}(n)$ for the set of all $n\times n$ matrices with entries in $\mathbb{R}$. 

\begin{definition}\label{definition:Clifford_module_action}
We call the module $V$ of the Clifford algebra $C\ell_{r,s}$ {\it admissible} if it carries a non-degenerate symmetric bilinear form $\langle \cdot\,, \cdot \rangle_V$ which satisfies  the following properties:
\begin{align}
\langle J_zX,J_zY \rangle_V
&= \langle z,z\rangle_{r,s}  \langle X,Y\rangle_V,\label{GL_1}\\
\langle J_zX, Y \rangle_V&=-\langle X, J_zY \rangle_V, \label{GL_2}\\
J_z^2&=- \langle z,z \rangle_{r,s} I, \label{GL_3}
\end{align}
where $I$ denotes the identity operator. Note that conditions (\ref{GL_1}) and (\ref{GL_2}) are equivalent if property (\ref{GL_3}) holds. The existence of admissible modules for the Clifford algebra 
$C\ell_{r,s}$ was shown in~\cite[Theorem 2.1]{Ci}. The following is known, see~\cite[Proposition 2.2]{Ci}: 
\end{definition}

\begin{lemma}
If $s>0$, then $(V, \langle \cdot \,, \cdot \rangle_V)$ has positive definite and negative definite subspaces of the same dimension. In particular, we have $\dim V = 2n$ for some $n \in \NN$. 
\end{lemma}
In what follows we consider the case $s>0$ only. We may define a Lie bracket $[\cdot\,, \cdot ]: V \times V \rightarrow \mathbb{R}^{r,s}$ through the following equation
\begin{equation*}
\langle J_zX, Y \rangle_V= \langle z, [X,Y] \rangle_{r,s}, \qquad  z \in \mathbb{R}^{r,s}, \; \; X,Y \in V. 
\end{equation*}
\begin{definition}
Let $V$ be an admissible $C\ell_{r,s}$-module. With the above bracket relation and centre $\mathbb{R}^{r,s}$ the space 
\begin{equation}\label{Def_Pseudo_H_type_Lie_algebra}
\mathcal{N}_{r,s} := V \oplus \mathbb{R}^{r,s}
\end{equation}
defines a $2$-step nilpotent Lie algebra which we call {\it pseudo $H$-type algebra}. Further  information on the algebraic properties of such algebras and an analysis of associated second order 
differential operators can be found in \cite{BFI,Ci, FM1}.
\end{definition}
\begin{example}
\label{Example_Heisenberg_Lie_algebra} (The Heisenberg algebra $\mathcal N_{0,1}$)
We represent the Heisenberg Lie algebra as a pseudo $H$-type algebra via an admissible $C\ell_{0,1}$-module. Let  be $z \in \mathbb{R}^{0,1}$ such that  $\langle z,z\rangle_{0,1}=-1$ 
and consider $V := \RR^{n} \times \RR^n$ with basis $(v_1, \ldots, v_n, w_1, \ldots , w_n)$. Define 
\begin{align*}
	  \scal{v_i}{v_j}_V &= \delta_{ij} , \qquad   \scal{w_i}{w_j}_V =  - \delta_{ij} , \qquad \scal{v_i}{w_j} = 0 . 
\end{align*}
We put $J_z v_i = w_i$, $J_z w_i  = v_i $ 

and extend $J_z$ by linearity. One obtains: 
\begin{equation*}
	\scal{[v_i, w_i ]}{z}_{0,1}= \scal{J_z v_i}{w_i}_V =  \scal{J_z v_i}{J_z v_i}_V= \scal{z}{z}_{0,1} \scal{v_i}{v_i}_V = - 1 . 
\end{equation*}
Thus, $[v_i,w_i] = z$, whereas the other commutators vanish. 
\end{example}
\begin{example}(The algebra $\mathcal{N}_{1,1}$) 
We choose a basis $\{z_1, z_2\}$ of $\mathbb{R}^{1,1}$ with 
\begin{align*}
  \scal{z_1}{z_1}_{1,1} =1 , \qquad  \langle z_2,z_2\rangle_{1,1} =-1 , \qquad  \langle z_1, z_2 \rangle_{1,1}   =0. 
\end{align*}
Then $C\ell_{1,1} \cong \mathbb{R}(2)$ and an admissible module $V$ of minimal  dimension is 4-dimensional. In this case we may  choose $v \in V$ with $\langle v,v \rangle_V=1$ and put 
\begin{align*}
X_1&=v, \qquad  X_2  = J_{z_1}v  , \qquad  X_3 = J_{z_1}J_{z_2}v,   \qquad   X_4= J_{z_2}v. 
\end{align*}
We  obtain the following table of commutation relations of $\mathcal{N}_{1,1}$: 
\[
\begin{array}{|c||c|c|c|c|}\hline
{\text{\tiny{\tiny[row,column]}}}   &X_1 &X_2&X_3&X_4\\\hline\hline
X_1&0&z_1 & 0  &z_2   \\\hline
X_2& -z_1 &0 & - z_2& 0 \\\hline
X_3& 0 & z_2 &0&z_1 \\\hline
X_4& - z_2 & 0&-z_1 &0\\\hline
\end{array}
\]
Moreover, we may show that   
\begin{align*}
	\langle X_2, X_2 \rangle_V&=1, & \langle X_3, X_3 \rangle_V&=-1, & \langle X_4, X_4 \rangle_V&=-1.
\end{align*}
\end{example}
For a 2-step nilpotent Lie algebra $\mathfrak g = T_eG$ induced by a connected and simply connected Lie group $G$, the   exponential map  $\exp: \mathfrak g  \rightarrow G$ becomes a diffeomorphism 
and therefore allows us to identify $\mathfrak{g}$ and $G$. From the Baker-Campbell Hausdorff formula, which states for a 2-step nilpotent Lie algebra that 
\begin{equation*}
	\exp(X) \ast  \exp(Y) = \exp \Big{(} X+Y+ \frac{1}{2} \big{[}X,Y \big{]} \Big{)}, 
\end{equation*}
 we may reconstruct the group structure. In what follows we will write $G_{r,s}$ for the connected, simply connected nilpotent Lie group with Lie algebra $\mathcal{N}_{r,s}$ and call it 
 {\it pseudo $H$-type group}. Note that with the above identification and as a vector space we have 
 $$G_{r,s}\cong \mathcal{N}_{r,s}  \cong \RR^{2n+r+s}.$$
 On the pseudo $H$-type algebra (\ref{Def_Pseudo_H_type_Lie_algebra}) we may define a scalar product by 
\begin{equation*}
\big{\langle} x+z, x^{\prime}+z^{\prime} \big{\rangle}_{\mathcal{N}_{r,s}} := \langle x, x^{\prime} \rangle_V+ \langle z,z^{\prime} \rangle_{r,s} \hspace{2ex} \mbox{\it where} \hspace{2ex} 
x+z , x^{\prime}+ z^{\prime} \in V \oplus \mathbb{R}^{r+s}
\end{equation*}
and extend it to a left-invariant  pseudo-Riemannian metric on $G_{r,s}$. 
\vspace{1ex}\par 
We now write up the construction more explicitly. Given a pseudo $H$-type algebra $\mathcal{N}_{r,s}= V \oplus \RR^{r,s}$ with  $V = \mathrm{span}  \{X_j : j = 1,\ldots, 2n \}$ and $\RR^{r,s} = \mathrm{span}  \{Z_k : k = 1,\ldots, r+s \}$ 
 we may identify the generators $X_j$ of a complement of the center and $Z_k$ of the center, respectively, with left-invariant vector fields on $G_{r,s} \cong \RR^{2n+r+s}$ as follows: 
\begin{equation}\label{eq:vector_fields_X_j_Z_k}
	X_j := \frac{\partial}{\partial x_j} +  \sum_{m=1}^{2n}  \sum_{k=1}^{r+s}    a_{mj}^k x_m \frac{\partial}{\partial z_k} , \quad j = 1, \ldots, 2n, \qquad 
	Z_k := \frac{\partial}{\partial z_k} , \quad k = 1, \ldots , r+s . 
\end{equation}
We assume that 
\begin{align*}
\langle X_i, X_j\rangle_V
&= \delta_{ij}, && i,j \in \{ 1, \ldots, n\},\\
\langle X_i, X_j \rangle_V
&= - \delta_{ij}, && i,j \in \{ n+1, \ldots, 2n\}\\
\langle X_i, X_j\rangle_V
&=0&&  i\in  \{ 1, \ldots, n\},\ \ j \in \{ n+1, \ldots, 2n\}, 
\end{align*}
The structure constants satisfy $a^k_{ij}=-a^k_{ji}$ and are defined through the commutation relations, i.e., we have
\begin{equation*}
[X_i, X_j]= 2 \sum_{k=1}^{r+s}  a_{ij}^k Z_k , \qquad i,j \in \{ 1, \ldots, 2n\} . 
\end{equation*}
In what follows we form the matrices $\Omega_k \in \mathbb{R}(2n)$ with entries $(\Omega_k)_{ij} := a_{ij}^k$. We also denote by $\langle \cdot\,,\cdot\rangle$ the standard Euclidean inner product 
on $\mathbb{R}^{2n}$. Then the corresponding Lie  group action  on $G_{r,s}\cong \RR^{2n+r+s}$ is given by 
\begin{equation}\label{GL_product_Lie_group}
(x, z) \ast (y , w ) = \biggr( x + y, z + w + \sum_{k=1}^{r+s} \scal{\Omega_k^T x}{y} \mathbf e_k \biggr), 
\end{equation}
 where $\mathbf e_k \in \RR^{r+s} $ is the $k$-th canonical unit vector. In what follows we put 
\begin{align*}
 	\Omega ( \eta) : = \eta_1 \Omega_1 + \ldots +  \eta_{r+s} \Omega_{r+s} , \qquad \eta \in \RR^{r+s}  . 
\end{align*}
Note that $\Omega ( \eta)^T = - \Omega(\eta)$. With the obvious notation the vector fields $X_j$ can be expressed in the form:  
\begin{equation}\label{eq:representation_X_j_Omega}
	X_j = \frac{\partial}{\partial x_j} -  \biggr( \Omega \Bigr(\frac{\partial}{\partial z_1} , \ldots, \frac{\partial}{\partial z_ {r+s} } \Bigr) x\biggr)_j . 
\end{equation}
In what follows we want to consider the second order differential operator $\Delta_{r,s}$ defined in (\ref{Definition_Ultra_hyperbolic_operator}) and 
associated with the pseudo $H$-type group $G_{r,s}$. We call $\Delta_{r,s}$ an {\it ultra-hyperbolic operator}, due to its similarly with the classical ultra-hyperbolic operator $\mathcal{L}$ in (\ref{Definition_classical_UHO_introduction}) acting on $\mathbb{R}^{2n}$ (see \cite{Hoermander,DRham}). Our aim is to calculate a fundamental solution for this operator, i.e., 
we look for a distribution $K_{r,s} \in \mathcal{S}^{\prime}(\mathbb{R}^{2n+r+s})$, which satisfies 
$$
	\Delta_{r,s} K_{r,s} = \delta_0 , 
$$
where $\delta_0$ is the Dirac distribution centered at $0 \in \RR^{2n+r+s}$. Note that different from $\mathcal{L}$ the operator $\Delta_{r,s}$ does not have constant coefficients and therefore even 
the existence of such a fundamental solution is not guaranteed. The differential operator $\Delta_{r,s}$ is built from left-invariant vector fields and therefore left-invariant itself. Let $g \in G_{r,s}$ be 
arbitrary and $K$ a fundamental solution of $\Delta_{r,s}$ as above. With the left-translation: 
\begin{equation*}
L_g: \mathcal{S}(\mathbb{R}^{2n+s+r}) \rightarrow \mathcal{S}(\mathbb{R}^{2n+r+s}): \varphi \mapsto L_g(\varphi)= \varphi( g * \cdot ) 
\end{equation*}
consider the distribution $K_g:= K \circ L_g$. Then by the last remark $K_g$ solves the equation
\begin{equation*}
\Delta_{r,s} K_g=\delta_g, 
\end{equation*}
where $\delta_g(\varphi)= \varphi(g)$ denotes the evaluation in $g$. In the formulas below we will observe a close relation between $\Delta_{r,s}$ and $\mathcal{L}$ similar to the relations between the sub-Laplacian $\Delta_{\textup{sub}}:=-\sum_{j=1}^{2n}  X_j^2$ on 2-step nilpotent Lie groups and the Laplacian on $\mathbb R^{2n}$.
\section{The ultra-hyperbolic operator}
\label{Section_2}
Our approach is based on  a  formal observation by J. Tie in \cite{Tie} where  a fundamental solution of the ultra-hyperbolic operator was constructed in the special case of the {\it Heisenberg Lie algebra} 
$\mathcal N_{0,1} \cong  \RR^{2n} \oplus \RR^{0,1}$ (cf. Example \ref{Example_Heisenberg_Lie_algebra} ). Here the vector fields in (\ref{eq:vector_fields_X_j_Z_k}) take the form:  
\begin{align*}
	X_j &= \frac{\partial}{\partial x_j} - \frac{x_{j+n}}2 \frac{\partial}{\partial z} , \qquad \mbox{\it and}\qquad
	X_{j+n} = \frac{\partial}{\partial x_{j+n}} + \frac{x_{j}}2 \frac{\partial}{\partial z}, \hspace{4ex} j=1, \ldots ,n. 
\end{align*}
More precisely, it is shown that 
\begin{equation}\label{eq_fund_solution_rs}
K_{r,s}(x, z) = \frac{i^{n+1} \Gamma(n)}{8\pi^{n+1}} \int_\RR \frac{1}{\left \{ \frac{\eta}4 (\sum_{j=1}^n x_j^2 - x_{j+n}^2 ) \coth \left( \frac{\eta}4  \right) - z \eta \right\}^n  } \left( \frac{\eta}{4 \sinh
 	\left( \frac{\eta}4  \right)} \right)^n \; \dd \eta  
\end{equation}
is a fundamental solution for (\ref{Definition_Ultra_hyperbolic_operator}). Note that $K_{r,s}(x, z)$ needs to be regularized, however a suitable regularization is shown to converge in the space of tempered distributions. Moreover, the structure of the singular support was discussed in \cite{Tie} and the  construction of  $K_{r,s}(x,z)$ was reduced to  determine the inverse symbol of $\Delta_{r,s}$  in the 
framework of a pseudo-differential calculus. However, Formula \eqref{eq_fund_solution_rs} may also be deduced by a formal change of coordinates. Putting 
\begin{align}\label{eq:variable_transform}
	\left\{\begin{array}{rl} y_j &= - i x_j , \quad j \in \{ 1, \ldots, n \} , \\
	y_{j+n} &= x_{j+n} , \quad j \in \{ 1, \ldots, n \} ,\\
	\qquad w &= -  iz , \end{array} \right. 
\end{align}
we  obtain for $j \in \{ 1, \ldots, n \}$ that 
\begin{align*}
 	\frac{\partial}{\partial x_j} = \frac{\partial y_j }{\partial x_j}  \frac{\partial}{\partial y_j} = 
	- i  \frac{\partial}{\partial y_j}  , \qquad 
 	\frac{\partial}{\partial x_{j+n}} = \frac{\partial}{\partial y_{j+n} } , \qquad \frac{\partial}{\partial z} = - i \frac{\partial}{\partial w}. 
\end{align*}
Then the operator $\Delta_{r,s}$ transforms formally to the corresponding hypoelliptic sub-Laplacian 
\begin{align*}
	\Delta_{\mathrm{sub}} := \sum_{j=1}^n - Y_j^2 - Y_{j+n}^2 , 
\end{align*}
where we have 
\begin{align*}
	Y_j &= \frac{\partial}{\partial y_j} - \frac{y_{j+n}}2 \frac{\partial}{\partial w} , \qquad Y_{j+n} = \frac{\partial}{\partial y_{j+n}} + \frac{y_{j}}2 \frac{\partial}{\partial w}, \qquad  j \in \{1, \ldots, n \}. 
\end{align*}
As is well known (e.g. \cite{BGP}) a fundamental solution of the sub-Laplacian $\Delta_{\mathrm{sub}}$ is given by 
\begin{equation*}\label{eq_fund_solution_sub}
 	K_{\mathrm{sub}}(y, w) := \frac{\Gamma(n)}{8\pi^{n+1}} \int_\RR \frac{1}{\left \{ \frac{\eta}4 (\sum_{j=1}^n y_j^2 + y_{j+n}^2 ) \coth \left( \frac{\eta}4  \right) +  iw \eta \right\}^n } \left( \frac{\eta}{4 \sinh
 	\left( \frac{\eta}4  \right)} \right)^n \; \dd \eta. 
\end{equation*}
Applying again \eqref{eq:variable_transform} and transforming $\eta \mapsto - \eta$  we obtain \eqref{eq_fund_solution_rs}  up to the factor $i^{n+1}$, 
which formally may be interpreted as a {\it ''complex Jacobian''}. 
\vspace{1ex}\par 
In the setting of a  general pseudo $H$-type group we aim to use a similar approach to obtain a fundamental solution for the ultra-hyperbolic operator. For ease of notation we work with 
the full symbol of the operator, which according to (\ref{eq:representation_X_j_Omega}) is given by:
\begin{align*}
	\sigma(\Delta_{r,s})(x, z,  \xi, \eta)
	&=\sum_{j=1}^n - \Big{\{} \xi_j- (\Omega(\eta)x)_j \Big{\}}^2 +  \Big{\{} \xi_{j+n}- (\Omega(\eta)x)_{j+n} \Big{\}}^2  .
\end{align*}
As a first step we start by listing some properties of the  matrix $\Omega(\eta)$ that reflect the pseudo $H$-type structure. To this end we consider for $\eta \in \RR^{r,s}$  the matrix $\rho(\eta) \in 
\mathbb{R}(2n)$ describing the Clifford module action, i.e., we have 
$$ J_{Z(\eta)} X_j = \sum_{l=1}^{2n} \rho(\eta)_{lj} X_l, \quad  \text{\it where} \quad Z(\eta) = \eta_1  Z_1 + \ldots + \eta_{r+s} Z_{r+s} . $$
Definition \ref{definition:Clifford_module_action}, (3) implies that: 
\begin{equation}\label{eq:property_rho_eta}
\rho(\eta)^2  = - \scal{\eta}{\eta}_{r,s} I.
\end{equation}
We put $V_+ := \mathrm{span} \{ X_1 , \ldots, X_n\}$ and $V_- := \mathrm{span} \{ X_{n+1} , \ldots, X_{2n}\}$, where $\langle\cdot\,,\cdot\rangle_V$ is positive definite on $V_+$ and negative 
definite on $V_-$. In what follows we call the elements of $V_+$ positive and of $V_-$ negative. Notice that $\rho=\rho(\eta)$ is linear with respect to $\eta$ by the definition of $J_{Z(\eta)}$. With the 
identity element $I \in \mathbb{R}(n)$ we define the matrix
\begin{equation*}\label{Definition_tau}
\tau:= 
\left(
\begin{array}{cc}
I & 0 \\
0 & -I 
\end{array}
\right)
\in \mathbb{R}(2n).  
\end{equation*}
Let us  express elements  $x,y \in V=V_+ \oplus V_-$ and linear maps $A$ on $V$ with respect to the coordinates $X_j$, $j=1, \ldots, 2n$. Then we obtain with the euclidean inner product 
$\langle \cdot\,, \cdot \rangle$: 
\begin{equation*}
\big{\langle} Ax,y \big{\rangle}_V
= \big{\langle} Ax,\tau y \big{\rangle}
= \big{\langle} x, A^T \tau y\big{\rangle} 
= \big{\langle} x, \tau A^T \tau y \big{\rangle}_V. 
\end{equation*}
Hence we observe that the transpose $A^*$ of $A$ with respect to the non-degenerate bilinear form $\langle \cdot\,, \cdot \rangle_V$ is given by 
\begin{equation}\label{eq:transpose_with_respect_to_V}
A^*= \tau A^T \tau. 
\end{equation}
\par 
The properties of the Clifford module action $J_z$ in Definition \ref{definition:Clifford_module_action} together with the identity (\ref{eq:transpose_with_respect_to_V}) imply that  
$\rho(\eta)=- \rho(\eta)^*$ for all $\eta \in \mathbb{R}^{r,s}$ and therefore:  

\begin{lemma}\label{Lemma_properties_matrices_A_B_C_D}
	For $\eta \in \RR^{r,s}\cong \mathbb{R}^r \times \mathbb{R}^s$ we have 
	\begin{align*}
		\tau \rho(\eta)  = - (\tau \rho(\eta) )^T = - \rho (\eta)^T \tau. 
	\end{align*}
\end{lemma}
\begin{remark} \label{Remark:properties_A_B_D}
	Let $Z$ be chosen such that $\langle Z,Z \rangle_{r,s} >0$. According to Definition \ref{definition:Clifford_module_action}, (1) it follows that  $J_Z$ maps positive elements to 
	positive elements and negative elements to negative elements, i.e., $J_Z$ 
	leaves $V_+$ and $V_-$ invariant. Moreover, if $Z$ satisfies $\langle Z,Z\rangle_{r,s}<0$ then $J_Z$ maps $V_+$ to $V_-$ and vice versa.
	Together with the previous  lemma this implies  that for $\eta = (\eta_+, \eta_- ) \in \mathbb R^{r}\times\mathbb R^s$ the matrix $\rho(\eta)$ may be written as
\begin{equation} \label{equ:definition_matrix_rho}
	 \rho (\eta) = \begin{pmatrix} A(\eta_+ ) & B( \eta_- ) \\ B(\eta_-)^T &   D( \eta_+ )   \end{pmatrix} , 
\end{equation}	 
such that the following assertions hold true: 
\begin{enumerate}
	\item $A(\eta_+)$ and $D(\eta_+)$ are skew-adjoint, 
	\item $A(\eta_+)^2 = -| \eta_+ |^2 I= D(\eta_+)^2$ and  $B(\eta_-)^T B(\eta_-) = |\eta_-|^2 I = B(\eta_-) B(\eta_-)^T$, 
	\item $A(\eta_+) B(\eta_-) + B(\eta_-) D(\eta_+) = 0$ and $B(\eta_-)^TA(\eta_+)+D(\eta_+)B(\eta_-)^T=0$.
\end{enumerate}
Here $|\cdot|^2$ denotes the euclidean square norm of a vector. 
\end{remark}
For the coefficients $a_{\ell j}^k=\Omega(Z_k)_{\ell j}$ we have  
\begin{align*}
	2 a_{\ell j}^k \scal{Z_k}{Z_k}_{r,s} = \scal{[X_{\ell},X_j]}{ Z_k}_{r,s}  = \scal{J_{Z_k} X_{\ell}}{X_j}_V = \big(\rho(Z_k)\big)_{j\ell}\scal{X_j}{X_j}_V 
\end{align*}
and thus, we obtain for $\eta\in \RR^{r,s}$ that 
\begin{equation*}
\begin{cases}
\big(\Omega(\eta_+)\big)_{\ell j}=\frac{1}{2}\big(\rho(\eta_+)^T\big)_{\ell j} &\text{\it for}\quad j=1,\ldots,n, \\
\big(\Omega(\eta_-)\big)_{\ell j}=\frac{1}{2} \big(\rho(\eta_-)^T\big)_{\ell j} &\text{\it  for}\quad j=n+1,\ldots,2n, \\
\big(\Omega(\eta_-)\big)_{\ell j}=-\frac{1}{2} \big(\rho(\eta_-)^T\big)_{\ell j} &\text{\it  for}\quad j=1,\ldots,n, \\
\big(\Omega(\eta_+)\big)_{\ell j}=-\frac{1}{2} \big(\rho(\eta_+)^T\big)_{\ell j} &\text{\it for}\quad j=n+1,\ldots,2n. 
\end{cases}
\end{equation*}
As a consequence the relation between the matrices $\Omega(\eta)$ and $\rho(\eta)$ is given by: 
\begin{equation}\label{eq:relation_between_Omega_rho_first_place}
	\Omega(\eta) = \frac12 
		\big{[} \tau \rho(\eta)^T \big{]}, \hspace{4ex} \mbox{\it where} \hspace{4ex} \eta \in \mathbb{R}^{r,s}. 
\end{equation}
We collect some properties of the matrix $\Omega(\eta)$: 
\begin{lemma} \label{lemma:properties_of_Omega}
Let $\eta \in \mathbb{R}^{r+s} \cong \mathbb{R}^r \times \mathbb{R}^s$. Then 
\begin{itemize}
\item[(1)] The matrix $\Omega(\eta)$ fulfills the relations: 
\begin{equation*}
	\big{[} \tau \Omega(\eta) \big{]}^2= \big{[} \Omega(\eta) \tau \big{]}^2 =- \frac{\scal{\eta}{\eta}_{r,s}}{4} I. 
\end{equation*}
In particular, since $\Omega(\eta)$ is skew-symmetric we have: 
\begin{equation*}
	\Omega(\eta)^T \tau \Omega(\eta)= \Omega (\eta) \tau \Omega(\eta)^T= \frac{\scal{\eta}{\eta}_{r,s}}{4} \tau. 
\end{equation*}
\item[(2)] If $\scal{\eta}{\eta}_{r,s}>0$, then the matrices $\tau \Omega(\eta)$ and $\Omega(\eta)\tau$ only have purely imaginary eigenvalues 
\begin{equation*}
	\lambda_{\eta, \pm}= \pm i \frac{\sqrt{\scal{\eta}{\eta}_{r,s}}}{2}. 
\end{equation*}
Both eigenvalues $\lambda_{\eta,+}$ and $\lambda_{\eta,-}$ have multiplicity $n$. 
\end{itemize}
\end{lemma}
\begin{proof}
	From  (\ref{eq:property_rho_eta}) and the above relation between $\Omega(\eta)$ and $\rho(\eta)$ we have: 
	$$ \big{[} \tau \Omega(\eta) \big{]}^2 = \frac14 \big{[} \rho(\eta)^T \big{]}^2 = - \frac{\scal{\eta}{\eta}_{r,s}}{4} I $$
	and likewise we obtain the other assertions in (1). By (1)  the minimal polynomial of $\Omega(\eta)\tau$ and $\tau \Omega(\eta)$ has the form 
$$p_{\eta}(\lambda)= \lambda^2+ \frac{\scal{\eta}{\eta}_{r,s} }{4} = \left( \lambda+i \frac{\sqrt{\scal{\eta}{\eta}_{r,s}}}{2} \right) \left( \lambda- i \frac{\sqrt{\scal{\eta}{\eta}_{r,s}}}{2} \right) 
$$ 
showing the assertion in (2). 
\end{proof}
In our analysis we will make use of the Fourier transform with respect to different  groups of variables. First we determine an operator $\Delta_{r,s}(\eta)$ such that 
\begin{equation*}
    \mathcal{F}_{z \mapsto \eta} ( \Delta_{r,s} \varphi)(\eta) = \Delta_{r,s}(\eta) ( \mathcal{F}_{z \mapsto \eta} \varphi) , \qquad \varphi \in \mathcal{S}(\mathbb{R}^{2n+r+s}) 
\end{equation*}
holds, where $\mathcal{F}_{z \mapsto \eta}$ denotes the Fourier transform with respect to the $z$-variables. Let us express $\Delta_{r,s}$ in the above coordinates 
in (\ref{eq:vector_fields_X_j_Z_k}): 
\begin{align*}
    \Delta_{r,s}
    &= \sum_{j=1}^n \left\{\frac{\partial}{\partial x_j} + \sum_{m=1}^{2n} \sum_{k=1}^{r+s} a_{mj}^kx_m \frac{\partial}{\partial z_k}\right\}^2-  \left\{\frac{\partial}{\partial x_{j+n}} 
        + \sum_{m=1}^{2n} \sum_{k=1}^{r+s} a_{m,j+n}^k  x_m \frac{\partial}{\partial z_k}\right\}^2.  
\end{align*}
Thus, we have 
\begin{align*}
	\Delta_{r,s}(\eta) &= \sum_{j=1}^n \left\{ \frac{\partial}{\partial x_j} + i \sum_{k=1}^{r+s} \sum_{m=1}^{2n} a_{mj}^k \eta_k x_m 
		\right\}^2- \left\{ \frac{\partial}{\partial x_{j+n}} + i  \sum_{k=1}^{r+s} \sum_{m=1}^{2n} a_{m, j+n} \eta_k x_m \right\}^2 \\
    &= \mathcal L - \sum_{j=1}^n \left\{ \sum_{k=1}^{r+s} \sum_{m=1}^{2n} a_{mj}^k \eta_k x_m \right\}^2 +   
    \sum_{j=1}^n \left\{ \sum_{k=1}^{r+s} \sum_{m=1}^{2n} a_{m,j+n}^k \eta_k x_m  \right\}^2 \\
    &\quad + 2i \sum_{j=1}^n \left\{ \sum_{k=1}^{r+s} \sum_{m=1}^{2n} a_{mj}^k \eta_k x_m  \frac{\partial}{\partial x_j} -  
         \sum_{k=1}^{r+s} \sum_{m=1}^{2n} a_{m,j+n}^k \eta_k x_m  \frac{\partial}{\partial x_{j+n}} \right\} ,
\end{align*}
where $\mathcal{L}$ denotes the classical ultra-hyperbolic operator (\ref{Definition_classical_UHO_introduction}). Note that we have used that $a_{jj}^k=0$ due to the skew-symmetry 
of $\Omega_k$. In order to simplify the expression we use that 
\begin{align*}
    \sum_{k=1}^{r+s} \sum_{m=1}^{2n} a_{mj}^k \eta_k x_m =\big{(} \Omega(\eta)^T x\big{)}_{j} ,
\end{align*}
which implies
\begin{align*}
    \Delta_{r,s}(\eta) &=  \mathcal L - \sum_{j=1}^n \Bigl\{ \big{(}\Omega(\eta)^T x\big{)}_{j}^2 - \big{(}\Omega(\eta)^T x\big{)}_{j+n}^2 \Bigr\} + 2 i  \big{(}\tau \Omega(\eta)^T x\big{)} \cdot \nabla_x . 
\end{align*}
Applying Lemma \ref{lemma:properties_of_Omega} we can calculate the sum on the right hand side: 
\begin{align*}
     \sum_{j=1}^n \Bigl\{ (\Omega(\eta)^T x)_{j}^2 - (\Omega(\eta)^T x)_{j+n}^2 \Bigr\} &= 
     \big{(}\tau \Omega(\eta)^T x\big{)}  \cdot \big{(}\Omega(\eta)^T x\big{)}\\
     & = \big{(}\Omega(\eta) \tau \Omega(\eta)^T x \big{)} \cdot x 
     =  \frac14 \scal{\eta}{\eta}_{r,s} P(x) , 
\end{align*}
where we define:
\begin{equation*}
    P(x):=\langle \tau x,x \rangle= \sum_{j=1}^n x_j^2 - x_{j+n}^2. 
\end{equation*}
Finally, we note that $\tau \Omega(\eta)^T = - \frac12 \rho(\eta)^T$, which proves:
\begin{lemma}\label{Lemma_short_form_of_L}
With the above notation  we have:
\begin{equation}\label{eq:hat_Delta_r,s}
     \Delta_{r,s}(\eta) =  \mathcal L - \frac{\scal{\eta}{\eta}_{r,s}}{4} P(x) -  i  x^T \rho(\eta) \nabla_x. 
\end{equation}
\end{lemma}
Let now $\mathcal{F}$ denote the Fourier transform on the full space $\mathbb{R}^{2n+r+s}$. Consider the differential operator $\mathcal{G}_{r,s}$  on $\mathcal{S}(\mathbb{R}^{2n+r+s})$ 
defined through the equation: 
\begin{equation}\label{eq:definition_operator_G_r_s}
\mathcal{F} \circ \Delta_{r,s} = \mathcal{G}_{r,s} \circ \mathcal{F}. 
\end{equation}
Lemma \ref{Lemma_short_form_of_L} allows us to easily  calculate $\mathcal{G}_{r,s}$ in an explicit form: 
\begin{corollary}\label{cor:explicit_form_G_r_s}
The operator $\mathcal{G}_{r,s}$ with (\ref{eq:definition_operator_G_r_s})  acts on $\varphi= \varphi(\xi, \eta_+, \eta_- ) \in \mathcal{S}(\mathbb{R}^{2n+r+s})$ as:  
\begin{equation*}
\mathcal{G}_{r,s}\varphi= -P(\xi) \varphi + \frac{\scal{\eta}{\eta}_{r,s}}{4}  \mathcal{L} \varphi +i \xi^T \rho(\eta)^T \nabla_{\xi} \varphi, 
\end{equation*}
where $\mathcal{L}$ is the ultra-hyperbolic operator in (\ref{Definition_classical_UHO_introduction}) with respect to the variables $\xi \in \mathbb{R}^{2n}$. 
Both operators $\mathcal{G}_{r,s}= \mathcal{F} \circ \Delta_{r,s} \circ \mathcal{F}^{-1}$ and $\Delta_{r,s}$ are formally selfadjoint on $L^2(\mathbb{R}^{2n+r+s})$.  
\end{corollary}
\begin{proof}
The last statement follows from the observation that $\Delta_{r,s}$ is formed by $2n$ squares of skew-symmetric vector fields. 
\end{proof}
\section{From the Sub-Laplacian to the ultra-hyperbolic operator}
\label{Section_From_Sub-Laplacian_to_UHO}
 From the above considerations we can express the full symbol of the ultra-hyperbolic operator $\Delta_{r,s}$ as follows: 
\begin{align*}
	\sigma(\Delta_{r,s})(x, z,  \xi, \eta)
	&= - P(\xi) - \frac{\scal{\eta}{\eta}_{r,s}}{4} P(x) +    x^T  \rho(\eta)  \xi.   
\end{align*}
Throughout the paper  we use the decomposition $x= ( x_+ ,  x_-) \in \RR^n \times \RR^n$ and $z = (z_+, z_-) \in\mathbb R^r\times\mathbb R^s$. As in the case of the Heisenberg algebra 
we formally transform variables and put 
\begin{equation}\label{eq:variable_transform_general_case}
	\left\{ \begin{array}{rl} y_+ &= - i x_+ ,  \\  
	y_- &= x_- , \\
	w_+ &= z_+ , \\
	w_- &= - i z_- .
	\end{array} \right. 
\end{equation}
By  $(\zeta_+ , \zeta_-  , \vartheta_+, \vartheta_- ) = ( i\xi_+, \xi_-, \eta_+, i \eta_-) $ we denote the variables dual to $(y_+, y_- w_+, w_-)$. 
According to the matrix representation in (\ref{equ:definition_matrix_rho}) we can be interpreted $\rho(\eta)$ as a $\mathbb{C}$-linear family of linear maps on $\mathbb{C}^{2n}$. 
Writing the symbol $\sigma(\Delta_{r,s})$ in new coordinates, we obtain 
\begin{align*}
	\sigma(\Delta_{r,s})(y,w,\zeta , \vartheta)  
		  &= |\zeta|^2  +   \frac{|\vartheta|^2}{4} |y|^2  +  ( iy_+  , y_-)  \rho(\vartheta_+, - i \vartheta_- ) \begin{pmatrix} - i \zeta_+ \\ \zeta_- \end{pmatrix}   \\
		  &= |\zeta|^2  +   \frac{|\vartheta|^2}{4} |y|^2  + y^T \Xi(\vartheta)  \zeta , 
\end{align*}
where as before $|\cdot|$ denotes the Euclidean norm and 
\begin{align*}
      \Xi(\vartheta) :=   \begin{pmatrix} A(\vartheta_+ ) & B( \vartheta_- ) \\ -B(\vartheta_-)^T &   D(\vartheta_+ )    \end{pmatrix} , \qquad \vartheta = (\vartheta_+, \vartheta_-) . 
\end{align*}
Using the properties of the matrices $A,B,D$ in Remark \ref{Remark:properties_A_B_D}  we obtain that  $\widetilde{\Xi}(\vartheta):= \frac{1}{2}\Xi(\vartheta)$ is skew-symmetric and 
$\widetilde{\Xi}(\vartheta)^2 = - \frac14 | \vartheta|^2 \mathrm{I}$. As above we have 
\begin{align*}
    \sigma(\Delta_{r,s})(y,w,\zeta , \vartheta)  &=  \sum_{j=1}^{2n} \left\{ \zeta_j + \Big{(}\widetilde{\Xi}(\vartheta)^T  y\Big{)}_j \right\}^2 , 
\end{align*}
which is exactly the full symbol of the corresponding sub-Laplacian $ \Delta_{\mathrm{sub}}$ of a two-step nilpotent Lie group with structure matrix $\widetilde{\Xi}(\vartheta)$. 
In the next step we calculate a fundamental solution of the sub-Laplacian $\Delta_{\mathrm{sub}}$ from the well known expression of its heat kernel (see \cite{BGG,OvFuIw}).
\begin{remark}
Let $\frac{\partial}{\partial t} + \frac12 \Delta_{\mathrm{sub}}$ be the sub-elliptic heat operator. Then the heat kernel $$k : (0,\infty) \times \RR^{2n+r+s} \to \RR$$ uniquely  is defined 
by the conditions 
\begin{enumerate}
 	\item $(\frac{\partial}{\partial t} + \frac12 \Delta_{\mathrm{sub}}) k =0$ and 
 	\item $\lim_{t \downarrow 0} k(t, \cdot ) = \delta_{0}$, in the sense of distributions. 
\end{enumerate}
\end{remark}
In what follows we will write $k_t(\cdot)$ instead of $k(t, \cdot)$. For a 2-step nilpotent Lie group it is well known (see e.g. \cite{BGG} or \cite[Theorem 10.2.7]{OvFuIw}) that 
the heat kernel explicitly can be written explicitly as an integral:
\begin{equation*}
	k_t(y,w)= \frac{1}{(2\pi t)^{n+r+s}} \int_{\mathbb{R}^{r+s}} e^{- \frac{f(y,w,\vartheta)}{t}} W(|\vartheta|) \; \dd \vartheta , 
 \end{equation*}
where 
\begin{equation*}
	W(|\vartheta| ) = \sqrt{ \det \frac{\widetilde{\Xi} (i \vartheta)}{\sinh(\widetilde{\Xi} (i \vartheta))}}  
  = \left( \frac{\frac{|\vartheta|}{2}}{\sinh\big{(}\frac{|\vartheta|}{2}\big{)}}\right)^n,\label{W_r_=_0}
\end{equation*}
is the so-called {\it volume element} and 
\begin{align*}
	f(y,w, \vartheta)&=i \scal{\vartheta}{w} + \frac{|\vartheta|}{4} \coth\Big{(}\frac{|\vartheta|}{2}\Big{)} \sum_{j=1}^{2n} y_j^2 
\end{align*}
is the {\it action function}. If we put $p:=n+r+s-1$, then the integral 
\begin{align*}
	&K_{\mathrm{sub}} (y,w)	=  \int_0^{\infty} k_{2t}(y,w) \;\dd t  \\
	&=\frac{1}{(4\pi)^{p+1}}\int_0^{\infty} \int_{\mathbb{R}^{r+s}} \frac{1}{t^{p+1}}  
	     \exp \left\{ - \frac{1}{2t} \Big{[} i \langle \vartheta,w \rangle + \frac{|\vartheta|}{4} \coth \left( \frac{|\vartheta |}{2} \right) \sum_{j=1}^{2n} y_j^2 \Big{]} \right\} W(|\vartheta|) \; \dd \vartheta \; \dd t \notag
\end{align*}
gives formally a fundamental solution of the sub-Laplacian. Indeed, we have 
$$  \Delta_{\mathrm{sub}} K_{\mathrm{sub}} = \int_0^{\infty} \Delta_{\mathrm{sub}} k_{2t} \; \dd t   = 
	- \int_0^{\infty} \frac{\partial}{\partial t} k_{2t} \; \dd t =-  \Bigr[ k_{2t} \Bigr]_0^\infty= \delta_0 . $$
	\par 
In the following steps we will perform a series of formal calculations changing from real to complex variables in the above  integral representation of $K_{\mathrm{sub}} (y,w)$. Convergence of the 
integral expressions is not guaranteed in each step. However, this formal consideration produces a distribution which in Section \ref{section_3}  will rigorously be shown to be a fundamental solution 
of the ultra-hyperbolic operator when $r=0$ and $s>0$. If we use the change of variables (\ref{eq:variable_transform_general_case}) in $K_{\mathrm{sub}} (y,w)$, then we obtain the new kernel: 
\begin{align*}
       &K_{r,s}(x,z)=\\
       &=\frac{i^{n+r+s}}{(4 \pi)^{p+1}} \int_0^{\infty} \int_{\mathbb{R}^{r+s}} \frac{1}{t^{p+1}} \exp \left\{ -\frac{1}{2t} \Big{[} i \langle \vartheta_+,z_+ \rangle + \langle \vartheta_-,z_- \rangle - \kappa(|\vartheta|) P(x) \Big{]} \right\}
W(|\vartheta|) \; \dd \vartheta \; \dd t, 
\end{align*}
where 
\begin{equation}\label{eq:definition_P_kappa_vartheta}
     P(x):= \sum_{j=1}^n x_j^2-x_{n+j}^2\hspace{4ex} \mbox{\it and } \hspace{4ex} \kappa(|\vartheta|):= \frac{|\vartheta|}{4} \coth \left( \frac{|\vartheta|}{2} \right). 
\end{equation}
The factor $i^{n+r+s}$ in front of the integral should be interpreted as a {\it complex Jacobian}. Note that $\kappa(\rho)= \frac{\rho}{4} \coth(\frac{\rho}{2})$ continuously extends to $\rho=0$ if we define: 
$\kappa(0):=\frac{1}{2}$. 
\vspace{1mm}\par 
In what follows we restrict ourself to the case $r=0$. The existence of a fundamental solution of $\Delta_{r,s}$ for $r>0$ will be discussed in Section \ref{s_r_>_0}. We then have 
$\vartheta= \vartheta_-, z=z_- \in \mathbb{R}^s$ and 
\begin{equation*}
       K_{0,s}(x,z)
       =\frac{i^{n+s}}{(4 \pi)^{n+s}} \int_0^{\infty} \int_{\mathbb{R}^{s}} \frac{1}{t^{n+s}} \exp \left\{ -\frac{1}{2t} \Big{[}  \langle \vartheta, z  \rangle - \kappa(|\vartheta|) P(x) \Big{]} \right\}W(|\vartheta|) \; \dd \vartheta \; \dd t.
\end{equation*}
Take a function $\varphi \in \mathcal{S}(\mathbb{R}^{2n+s})$ and consider a formal integral:
\begin{equation*}\label{eq:integral_K_0_s_against_varphi}
     \begin{aligned}
     &\int_{\mathbb{R}^{2n+s}} \varphi(x,z) K_{0,s}(x,z) \; \dd x \; \dd z= \\
     &=\frac{i^{n+s}}{(4 \pi)^{n+s}} \int_0^{\infty} \int_{\mathbb{R}^s} \int_{\mathbb{R}^{2n+s}} \frac{1}{t^{n+s}} e^{ -\frac{1}{2t} \big{[}  \langle \vartheta, z  \rangle - \kappa(|\vartheta|) P(x) \big{]} }W(|\vartheta|)
       \varphi(x,z)  \; \dd x \; \dd z \; \dd \vartheta \; \dd t\\
     &=(*). 
     \end{aligned}
\end{equation*}
We change variables $\vartheta \rightarrow - \vartheta \in \mathbb{R}^s$ and formally we replace the integration over $t \in \mathbb{R}$ by an integration over $it$ where $t \in \mathbb{R}$:
\begin{equation*}
(*)=\frac{i}{(4 \pi)^{n+s}} \int_0^{\infty} \int_{\mathbb{R}^s} \int_{\mathbb{R}^{2n+s}} \frac{1}{t^{n+s}} e^{ \frac{i}{2t} \big{[}  \langle - \vartheta, z  \rangle - \kappa(|\vartheta|) P(x) \big{]}}W(|\vartheta|)
       \varphi(x,z) \; \dd x \; \dd z\;  \dd \vartheta \; \dd t. 
\end{equation*}
\par 
Now, we perform a Fourier transform $\mathcal{F}_{x \rightarrow \xi}$  in the last integral with respect to $x \in \mathbb{R}^{2n}$. Lemma
 \ref{Lemma_Fourier_transformLexpontion_quadratic_form} below provides a useful identity, which we apply in the calculation: 
\begin{lemma}\label{Lemma_Fourier_transformLexpontion_quadratic_form}
	Let $\beta_j \in \CC \setminus \{0\}$, $j=1,\ldots,2n$, with  $\textup{Re} (\beta_j) \ge  0$. Then 
	\begin{align*}
		\mathcal{F}^{-1}_{x \to \xi} \left[ \exp \left\{ - \frac12 \sum_{j=1}^{2n} x_j^2 \beta_j \right\} \right] (\xi) 
			&= \left( \prod_{j=1}^{2n}\frac{1}{\sqrt{\beta_j}} \right) \exp \left\{ - \frac12 \sum_{j=1}^{2n} \frac{\xi^2}{\beta_j} \right\} . 
	\end{align*}
	In this formula the holomorphic branch of the square root on $\CC \setminus (-\infty,0]$  is chosen such that $\sqrt{\beta_j}>0$ for $\beta_j > 0$.
\end{lemma}
In fact, the above identity follows easily for $\beta_j > 0$. By analytic continuation and continuity we observe that the equality holds true for $\textup{Re}(\beta_j) \ge 0$.  
In our case we set 
\begin{equation*}
      \beta_j:= i 
           \begin{cases}
                  \frac{ \kappa (|\vartheta|)}{t}, & \textup{\it if} \;\; j=1, \ldots, n, \\
                  - \frac{\kappa ( |\vartheta|)}{t}, & \textup{\it if} \: \:  j=n+1, \ldots, 2n. 
          \end{cases}
\end{equation*}
Note that $\beta_{j+n} = \overline{\beta_j}$ for $j=1, \ldots, n$ and thus, 
\begin{equation*}
\prod_{j=1}^{2n} \frac{1}{\sqrt{\beta_j}} =  \prod_{j=1}^{n} |\beta_j|^{-1} = |\beta_1|^{-n} =\frac{t^n}{ \kappa\big{(}|\vartheta|\big{)}^{n}}= \frac{(4t)^n}{|\vartheta|^n} \tanh^n \left( \frac{|\vartheta|}{2} \right). 
\end{equation*} 
Lemma \ref{Lemma_Fourier_transformLexpontion_quadratic_form} shows for all $t>0$: 
\begin{equation*}
\mathcal{F}^{-1}_{x \rightarrow \xi} \left[\exp \left\{-\frac{i\kappa(|\vartheta|) P(x)}{2t}\right\} \right] (\xi) 
=  \frac{t^n}{\kappa(|\vartheta|)^n} \exp \left\{ \frac{it}{2} \; \frac{P(\xi)}{\kappa(|\vartheta|)} \right\}. 
\end{equation*}
We use this formula on the right hand side of ($*$) to obtain: 
\begin{equation*}
(*)= \frac{i(2\pi)^{\frac{s}{2}}}{(4 \pi)^{n+s}} \int_0^{\infty} \int_{\mathbb{R}^{2n+s}} \frac{1}{t^s}  \frac{W(|\vartheta|)}{\kappa(|\vartheta|)^n} 
\exp \left\{ \frac{it}{2} \; \frac{P(\xi)}{\kappa(|\vartheta|)} \right\}\left[ \mathcal{F}
         \varphi \right]\Big{(}\xi,\frac{\vartheta}{2t}\Big{)} \; \dd  \xi \; \dd z \; \dd\vartheta \; \dd t,
\end{equation*}
where $\mathcal{F}$ denotes the Fourier transform on $\mathbb{R}^{2n+s}$. Recall that for $t >0$: 
\begin{equation*}\label{eq:formula_W_theta_divided_by_kappa}
    \frac{W(t)}{\kappa(t)^n} 
    = \left(\frac{\frac{t}{2}}{\sinh \big{(} \frac{t}{2} \big{)}} 
        \right)^n \; \frac{4^n}{t^n} \tanh\left(\frac{t}{2} \right)^n=\frac{2^n}{\cosh \big{(}\frac{t}{2} \big{)}^n}. 
\end{equation*}
\par 
We insert the last relation and change variables $\vartheta \rightarrow 2t\vartheta$ in the integration over $\mathbb{R}^s$ to obtain: 
\begin{equation*}
         (*)=\frac{i}{(2\pi)^{n+\frac{s}{2}}} \int_0^{\infty} \int_{\mathbb{R}^{2n+s}} \frac{[\mathcal{F}\varphi](\xi, \vartheta)}{\cosh \big{(} t|\vartheta| \big{)}^n} 
                 \exp\left\{ i \frac{\tanh(t |\vartheta|)}{|\vartheta|} \; P(\xi) \right\} \; \dd \xi \; \dd \vartheta \; \dd t. 
\end{equation*}
Finally, consider the change of variables $t \mapsto \frac{t}{|\vartheta|}$ in the integration over $\mathbb{R}_+$. We obtain the expression: 
\begin{multline*}
\int_{\mathbb{R}^{2n+s}} \varphi(x,z)K_{0,s}(x,z) \; \dd x \; \dd z=\\
=\frac{i}{(2\pi)^{n+\frac{s}{2}}} \int_{\mathbb{R}^{2n+s}} \int_0^{\infty}\frac{1}{|\vartheta|\cosh^n t}\exp \left\{ i \frac{\tanh t }{|\vartheta|} P(\xi) \right\} \; \dd t 
 [\mathcal{F}\varphi](\xi, \vartheta) \; \dd \xi \; \dd \vartheta. 
\end{multline*}
Hence let us define the kernel $q(\xi, \vartheta)$ for $\vartheta \ne 0$ by:
\begin{equation}\label{kernel_q}
q(\xi, \vartheta):=\frac{i}{(2\pi)^{n+\frac{s}{2}}} \int_0^{\infty}\frac{1}{|\vartheta|\cosh^n t }\exp \left\{ i \frac{\tanh t}{|\vartheta|} \: P(\xi) \right\} \;\dd t. 
\end{equation}
\section{The fundamental solution of $\Delta_{0,s}$}
\label{section_3}
Assuming  $s > 1$ we observe that the function $q$ in (\ref{kernel_q}) satisfies: 
$$q \in L^1\big{(}\RR^{2n+s}; ( 1+ |\xi|^2 + |\vartheta|^2)^{-N} \; \dd (\xi ,\vartheta)\big{)}$$ 
for sufficiently large $N >0$. In particular, it defines a tempered  distribution on $\RR^{2n+s}$ and we may define $K_{0,s} \in \mathcal S'(\RR^{2n+s})$. 
Let $\mathcal{F}$ denote the Fourier transform on $\mathbb{R}^{2n+s}$. Given a rapidly decreasing function $\varphi \in \mathcal S(\RR^{2n+s})$  we put  
\begin{align}
	 K_{0,s}\big{(}\varphi\big{)}
	:=\int_{\RR^{2n+s}} q(\xi, \vartheta) \big{[} \mathcal{F}\varphi\big{]} (\xi, \vartheta) \; \dd \xi \; \dd \vartheta . 
	\label{Definition_K_0_s}
\end{align}
\begin{theorem}\label{th:fundamental_solution}
	The distribution $K_{0,s}$ defines a fundamental solution of $\Delta_{0,s}$, i.e., we have $\Delta_{0,s} K_{0,s} = \delta_0$ or equivalently 
	\begin{align*}
		  K_{0,s} (\Delta_{0,s} \varphi)  = \varphi(0) , \hspace{4ex}  \mbox{\it for all} \hspace{4ex}  \varphi \in \mathcal S(\RR^{2n+s}) .
	\end{align*}
\end{theorem}
Since our derivation of Theorem \ref{th:fundamental_solution} in Section \ref{Section_From_Sub-Laplacian_to_UHO} was purely formal, we have to provide a rigorous proof. 
To this end we have to show that:  
\begin{equation*}
\int_{\RR^{2n+s}} q(\xi, \vartheta) \big{[} \mathcal{F} \Delta_{0,s} \varphi \big{]}(\xi, \vartheta) \; \dd \vartheta \; \dd \xi  =\varphi(0)=
	\frac{1}{(2\pi)^{n+s/2}} \int_{\RR^{2n+s}}  \big{[} \mathcal{F}\varphi\big{]} (\xi, \vartheta) \; \dd \xi \; \dd \vartheta . 
\end{equation*}
\par 
Let $\mathcal{G}_{0,s}$ be the differential operator on $\mathcal{S}(\mathbb{R}^{2n+s})$ defined through the equation (\ref{eq:definition_operator_G_r_s}). 
As was shown in Corollary \ref{cor:explicit_form_G_r_s} the operator $\mathcal{G}_{0,s}$ has the explicit form: 
\begin{equation*}
        \mathcal{G}_{0,s}   = - P(\xi) -  \frac{ |\vartheta|^2}{4}  \mathcal L +  i  \xi^T \rho(\vartheta)^T \nabla_\xi . 
\end{equation*}
Put 
\begin{equation*}
\overline{\mathcal{G}}_{0,s}   := - P(\xi) -  \frac{ |\vartheta|^2}{4}  \mathcal L -  i  \xi^T \rho(\vartheta)^T \nabla_\xi . 
\end{equation*}
As $q \in C^\infty(\RR^{2n} \times \RR^s \setminus \{0\})$ we have for each $\varphi \in \mathcal S(\RR^{2n+s})$ by partial integration: 
\begin{align*}
	\int_{\RR^{2n+s}} q(\xi, \vartheta)  \big{[}\mathcal{F} \circ  \Delta_{0,s} \varphi\big{]} &(\xi, \vartheta) \; \dd \vartheta \; \dd \xi =\\ 
	=&\lim_{\varepsilon \to 0} \int_{\RR^{2n}} \int_{|\vartheta|>\varepsilon} q(\xi, \vartheta)  \big{[}\mathcal{G}_{0,s} \circ \mathcal{F} \varphi\big{]} (\xi, \vartheta) \; \dd \vartheta \; \dd \xi \\
	=& \lim_{\varepsilon \to 0} \int_{\RR^{2n}} \int_{|\vartheta|>\varepsilon}\big{[}\overline{\mathcal{G}}_{0,s} q\big{]}(\xi, \vartheta)  \big{[} \mathcal{F}\varphi\big{]}  (\xi, \vartheta) \; \dd \vartheta \; \dd \xi. 
\end{align*}
\vspace{1ex}\\
Thus, the proof of Theorem \ref{th:fundamental_solution} follows from the next lemma. 
\begin{lemma}
	With the above notation we have 
	$[ \overline{\mathcal{G}}_{0,s} q](\xi, \vartheta) = (2\pi)^{-(n+s/2)}$. 
\end{lemma}
\begin{proof}
	We note that $q$ is of the form 
	$ q(\xi, \vartheta) = a \big{(}P(\xi), \vartheta \big{)}$ where $a = a(v,\vartheta)\colon \RR \times \RR^{s}  \to \CC$.
Then we obtain by a straightforward calculation: 
	\begin{equation*}\label{eq:differential_equation_L}
		\big{[}\mathcal L q\big{]} (\xi,\vartheta) = 4 P(\xi) \partial_{v}^2 a \big{(}P(\xi), \vartheta\big{)} + 4n \partial_{v}a\big{(}P(\xi), \vartheta \big{)} 
	\end{equation*}
Moreover, one has: 
\begin{equation*}	
		\xi^T \rho(\vartheta)^T \big{(}\nabla_\xi q\big{)}(\xi, \vartheta) = 2\big{\langle} \tau \rho(\vartheta) \xi,  \xi \big{\rangle} \partial_{v} a\big{(}P(\xi), \vartheta\big{)} = 0,
\end{equation*}
as $\tau \rho(\vartheta)=-2 \tau \Omega(\vartheta) \tau $ is skew-adjoint. Thus, we have 
	\begin{align*}
		\big{[}{\overline{\mathcal{G}}_{0,s}}q\big{]}(\xi,\vartheta)  &=  - P(\xi) a\big{(}P(\xi), \vartheta\big{)}  - n |\vartheta|^2  \partial_v a\big{(}P(\xi),\vartheta\big{)} -
		 |\vartheta|^2  P(\xi) \partial_v^2 a\big{(}P(\xi), \vartheta\big{)} . 
	\end{align*}
Using the expression of the function $a(v,\vartheta)$ gives: 
	\begin{align*}
		\big{[}\overline{\mathcal{G}}_{0,s}q\big{]}(\xi,\vartheta)  &= \frac{i}{(2\pi)^{n+\frac{s}{2}}} \int_0^{\infty}
		\left\{ - i n \tanh t  +  P(\xi) \frac{\tanh^2 t -1}{|\vartheta|}\right\} 
			\times \\
		&\hspace{30ex}  \times \frac{1}{\cosh^n t}\exp \biggl\{ i \frac{\tanh t}{|\vartheta|}  P(\xi) \biggr\} \;\dd t \\
		&= \frac{-1}{(2\pi)^{n+\frac{s}{2}}} \int_0^{\infty} \frac{\dd}{\dd t} \left[ \frac{1}{\cosh^n t }\exp \biggl\{ i \frac{\tanh  t }{|\vartheta|}  P(\xi) \biggr\} \right] \;\dd t \\
		&= \frac{1}{(2\pi)^{n+\frac{s}{2}}},  
	\end{align*}
and the assertion follows. 
\end{proof}
\section{A family of fundamental solutions}
\label{Section_non-uniqueness-of-the-fundamental-solution}
Motivated by the previous section we want to discuss now possible  restrictions on further fundamental solutions of the ultra-hyperbolic operator $\Delta_{0,s}$. To this end we define
$A := \{ {(\xi, \vartheta)} \in \RR^{2n+s} : P(\xi) \neq 0 \; \mbox{\it and} \;  |\vartheta| \neq 0 \}$  and we assume that  $\tilde K_{0,s}$ is a fundamental solution of $\Delta_{0,s}$, which satisfies
\begin{align*}
    \widetilde  K_{0,s}\big{(}\varphi\big{)}
	&=\int_{A}\tilde q(\xi, \vartheta) \big{[} \mathcal{F}\varphi\big{]} (\xi, \vartheta) \; \dd \xi \; \dd \vartheta 
\end{align*}
for all functions $\varphi \in \mathcal S(\RR^{2n+s})$ such that $\mathrm{supp}(\mathcal{F}\varphi) \subseteq A$. We additionally assume that $\tilde q \in C^\infty( A)$ and 
$$ \tilde q(\xi, \vartheta) = a (P(\xi), \vartheta) , \quad \text{\it for some }\quad  a =  a(v, \vartheta) \in C^\infty( (\RR \setminus \{0\}) \times (\RR^s \setminus \{0\})) .  $$
As $\widetilde K_{0,s}$ was assumed to be  a fundamental solution, we have $\overline{\mathcal{G}}_{0,s} \tilde q = (2\pi)^{-(n +\frac{s}2)}$ on $A$, where 
$\overline{\mathcal{G}}_{0,s}$ is given as above. As before this implies  
\begin{align*}
    (2\pi)^{-\frac{2n+s}2} 	&= - P(\xi) a\big{(} P(\xi),\vartheta\big{)}  - n |\vartheta|^2  \partial_v a\big{(}P(\xi),\vartheta \big{)} -
		 |\vartheta|^2  P(\xi) \partial_v^2 a\big{(}P(\xi),\vartheta \big{)}, 
\end{align*}
and thus, $a(\cdot, \vartheta)$ is  a solution of the differential equation of second-order 
\begin{align*}
	\gamma &= - v f(v)  -  n \alpha  f'(v) -  \alpha v  f''(v)  ,
\end{align*}
where $\gamma :=  (2\pi)^{-\frac{2n+s}2}$ and $\alpha := |\vartheta|^2$. The general solution may be calculated explicitly. To this end we use  the ansatz $f(v) = g( v /\sqrt{\alpha})$, which leads to  
$$
\frac{\gamma}{\sqrt{\alpha}}  = - v  g( v)  - n g'( v)  - v  g''( v). 
$$
Thus, for $g (v) = v^{\frac{1-n}2} h (v)$ we obtain 
\begin{align*}
	\frac{\gamma}{\sqrt{\alpha}}  
	&=  \left\{ - v^{\frac{1-n}2 +1}  - \left( \frac{1-n}2\right) \left( n - \frac{n+1}2  \right)  v^{\frac{1-n}2-1}  \right\} h (v) 
	-  v^{\frac{1-n}2}  h' (v)  - v^{\frac{1-n}2+1} h'' (v) ,
\end{align*}
and finally we have 
\begin{align*}
	- \gamma \: \frac{v^{\frac{n-1}2+1}}{\sqrt{\alpha}}  &=  {\left\{ v^2-\left(\frac{n-1}2\right)^2\right\}} h(v)  + v   h'  (v)  + v^2  h'' (v) . 
\end{align*}
The general solution of this differential equation is well-known and may be written as 
$$ h(v) =   \frac{i \gamma \: 2^{\frac{n-1}2 -1} \: \sqrt{\pi} \: \Gamma \left( \frac{n}2 \right) }{\sqrt{\alpha}}  \Bigl \{ c_1 J_{\frac{n-1}2}(v) + c_2 Y_{\frac{n-1}2} (v) + i  \mathbf H_{\frac{n-1}2}(v) \Bigr\} , $$
for coefficients $c_1, c_2 \in \CC$. Here $J_{\frac{n-1}2}$ and $Y_{\frac{n-1}2}$ are the {\it Bessel functions of the first and second kind} and $\mathbf H_{\frac{n-1}2}$ is the 
{{\it Struve function}, which is defined via the integral representation below. Recall that ${\bf H}_{\frac{n-1}{2}}$ solves the inhomogeneous Bessel equation
\begin{equation*}
\frac{4\left(\frac{v}{2}\right)^{\frac{n+1}{2}}}{\sqrt{\pi} \Gamma\big{(}\frac{n}{2} \big{)}}=\left\{ v^2- \left( \frac{n-1}{2} \right)^2 \right\} h(v)+v h^{\prime}(v)+v^2 h^{\prime \prime}(v), 
\end{equation*}
whereas $J_{\frac{n-1}2}$ and $Y_{\frac{n-1}2}$ solve the corresponding homogenous Bessel equation.} This implies
\begin{align*}
	\tilde q(\xi, \vartheta) &=  \frac{i \; \sqrt{\pi} \; \Gamma \left( \frac{n}2 \right)}{2(2\pi)^{n+s/2} |\vartheta|}  
	\left(\frac{2|\vartheta|}{P(\xi)} \right)^{\frac{n-1}2} \biggr\{ c_1 (\vartheta) J_{\frac{n-1}2} \left( \frac{ P(\xi)}{|\vartheta|} \right) + c_2 (\vartheta) Y_{\frac{n-1}2} \left( \frac{P(\xi)}{|\vartheta|} \right) \\
	&\hphantom{= \frac{i \cdot \sqrt{\pi} \; \Gamma \left( \frac{n}2 \right)   }{2(2\pi)^{n+s/2} |\vartheta|}  
	\left(\frac{2|\vartheta|}{P(\xi)} \right)^{\frac{n-1}2} \biggr\{} + i \mathbf H_{\frac{n-1}2}\left( \frac{P(\xi)}{|\vartheta|} \right) \biggr\}
\end{align*}
for measurable functions $c_1, c_2 : \RR^{s} \to \CC$. 

To recover the fundamental solution  from the previous section we use for $\nu \ge \frac12$ the well-known integral representation for the Bessel function and the Struve function 
(cf. formula 12.1.6 of Chapter 12 in \cite{AbSt}): 
\begin{align*}
	J_\nu (v) &= \frac{2 \left(\frac{v}2\right)^\nu }{\sqrt{\pi} \Gamma\left( \nu + \frac12 \right) } \int_{0}^1 ( 1 - \rho^2)^{\nu -\frac{1}2} \cos(v\rho) \; \dd \rho ,\\
	{\mathbf H}_{\nu}(v) &= \frac{2 \left(\frac{v}2\right)^\nu }{\sqrt{\pi}\Gamma\left( \nu + \frac12 \right)} \int_{0}^1 ( 1 - \rho^2)^{\nu -\frac{1}2} \sin (v\rho) \; \dd \rho .
\end{align*}
The transformation $\rho = \tanh  t$ gives us $\frac{\dd \rho}{\dd t } =  1- \rho^2$, $\cosh^2 t = \frac{1}{1-\rho^2}$, and thus, we obtain 
\begin{align*}
	 J_{\frac{n-1}2} (v) +  i {\mathbf H}_{\frac{n-1}2} (v) &= \frac{2 \left(\frac{v}2\right)^{\frac{n-1}2} }{\sqrt{\pi}  \Gamma\left( \frac{n}2 \right)  } 
		\int_0^\infty \frac{\exp\{ { iv  \tanh t}  \} }{\cosh^n t} \; \dd t . 
\end{align*}
Thus, the case $c_1 \equiv 1$ and $c_2 \equiv 0$ will give the fundamental solution $q=\tilde{q}$ in (\ref{kernel_q}) which was obtained in the previous section. 

Conversely, we may provide functions $c_1 , c_2 : \RR^s \to \CC$ and try to obtain further fundamental solutions. To this end we assume that $c_1$ is (for the sake of simplicity) bounded and that $c_2 \equiv 0$.  Define   $\lambda := \frac{c_1 + 1}2$, $\mu :=  \frac{1 - c_1 }2$ and let 
\begin{align}\label{Fundamental_solution_general_form_lambda_mu}
	q^{\lambda, \mu}_{0,s}(\xi, \vartheta) &:= \frac{i \; \sqrt{\pi} \; \Gamma \left( \frac{n}2 \right)}{2(2\pi)^{n+s/2} |\vartheta|}  
	\left(\frac{2|\vartheta|}{P(\xi)} \right)^{\frac{n-1}2} \biggr\{ c_1(\vartheta) J_{\frac{n-1}2} \left( \frac{ P(\xi)}{|\vartheta|} \right) + i \mathbf H_{\frac{n-1}2}\left( \frac{P(\xi)}{|\vartheta|} \right) \biggr\} \\
	&= \frac{i}{(2\pi)^{n+s/2} |\vartheta|}  \int_{0}^1 ( 1 - \rho^2)^{{\frac{n-2}{2}}} \Bigl\{ \lambda(\vartheta) e^{i {\frac{P(\xi)}{|\vartheta|} }\rho} -  
	\mu(\vartheta) e^{-i {\frac{P(\xi)}{|\vartheta|}}\rho} \Bigr\} \; \dd \rho .\notag
\end{align}
Then 
\begin{align}\label{eq:def_K^pm} 
    K^{\lambda, \mu}_{0,s}(\varphi) := \int_{\mathbb{R}^{2n+s}} q^{\lambda, \mu}_{0,s}(\xi, \vartheta) \big{[} \mathcal{F} \varphi\big{]}(\xi, \vartheta) \; \dd(\xi, \vartheta) 
\end{align}
is well-defined for arbitrary $\varphi \in \mathcal S(\RR^{2n+s})$ and  it gives a tempered distribution $K^{\lambda, \mu}_{0,s} \in \mathcal S'(\RR^{2n+s})$. As in the previous chapter we obtain:
\begin{theorem}\label{theorem_family_of_fundamental_solutions}
    For any pair of bounded functions $\lambda, \mu : \RR^s \to \CC$ which satisfies $\lambda + \mu \equiv 1$ we have that $K_{0,s}^{\lambda, \mu}$ defined by \eqref{eq:def_K^pm} is a fundamental solution of $\Delta_{0,s}$. 
\end{theorem}
The case $c_2 \not \equiv 0$ is more involved since the  Bessel function of the second kind has the following asymptotic behavior
\begin{align*}
    \left(\frac{v}2\right)^{- \frac{n-1}2}  Y_{\frac{n-1}2} (v)  = - \frac{\Gamma\left(\frac{n-1}2\right)}{\pi} \left(\frac{v}2\right)^{-(n-1)} + \mathcal O(v^{-(n-1)+1})  , \qquad v \to 0, \; n > 1 . 
\end{align*}
\vspace{1ex}\\
Thus, a priori it is not clear how to define $\widetilde K_{0,s}(\varphi) $ for arbitrary $\varphi \in \mathcal S(\RR^{2n+s})$, and once it has been defined this  distribution not necessarily gives 
a fundamental solution. 
\begin{example} {We consider the case $c_1 \equiv 0$, $c_2 \equiv - i$, $n=2$ and we use the formula  
\begin{align}
	  {{\bf H}_{\nu}(v)} - Y_{\nu} (v) = \frac{2 \left(\frac{v}2\right)^\nu}{\Gamma\left(\nu + \frac{1}2 \right) \Gamma\left(\frac{1}2 \right)}
	  \int_0^{\infty} e^{-v \rho} (1+\rho^2)^{\nu - \frac12} \; \dd \rho , \label{GL_relation_difference_H_Y}
\end{align}
{(cf. formula 12.1.8, Chapter 12 of \cite{AbSt})}.  Note that $\nu=\frac{n-1}{2}=\frac{1}{2}$ is a half-integer and therefore $Y_{\frac{1}{2}}(v)$ is not defined for negative values of $z$. 
In particular, the function (\ref{GL_relation_difference_H_Y}) has a pole at $z=0$. In the case where $P(\xi)>0$, we have $v=\frac{P(\xi)}{|\vartheta|}>0$ and using (\ref{GL_relation_difference_H_Y}) 
in the expression of $\tilde{q}(\xi, \vartheta)$ above gives: 

\begin{align*}
    \tilde q(\xi, \vartheta) = \frac{- 1}{(2\pi)^{2+s/2}} \frac1{P(\xi)} , \qquad (\xi , \vartheta) \in \RR^4 \times \RR^s, \qquad P(\xi)>0.
\end{align*}
A possible candidate for a fundamental solution of $\tilde K_{0,s}$ would be 
\begin{equation}\label{candidate_for_a_fundamental_solution_n=2}
            \widetilde K_{0,s}(\varphi)  = \frac{-1}{(2\pi)^{2+s/2}}  \int_{\RR^{s}}  \frac{1}{P} \Bigl[ \mathcal{F}\varphi  (\cdot, \vartheta) \Bigr]  \; \dd \vartheta  , 
\end{equation}
where $\frac{1}{P}$ is interpreted as a distribution on $\mathbb{R}^4$. 
}
\end{example}
{
We finally discuss a possible choice of the distribution that appears in the last example. According to \cite{Gelfand_Shilov} we can interpret $\frac{1}{P(\xi)}$ acting on suitable test functions $\varphi$ 
on $\mathbb{R}^{2n}$ in different ways. With $\lambda \in \mathbb{C}$ such that $\textup{Re}(\lambda)>0$ and using the notation of  \cite{Gelfand_Shilov} we consider 
\begin{equation}
\label{Definition_P_+_-Gelfand}
\big{(} P_{\pm}^{\lambda}, \varphi \big{)} := \int_{\pm P(\xi)>0} (\pm P)^{\lambda} \varphi \; \dd \xi \hspace{3ex} \mbox{\it and} \hspace{3ex} (P\pm i0)^{\lambda} :
=P_+^{\lambda}+e^{\pm \pi \lambda i} P_-^{\lambda}. 
\end{equation}
The following results can be found in Chapter 12 of \cite{Gelfand_Shilov}: 
\begin{proposition}\label{Propoosition_meromorphic_extension_distribution}
The maps $\lambda \mapsto ((P\pm i0)^{\lambda}, \varphi)$ admit a meromorphic extension to the complex plane with only simple poles at most at $-n,-n-1,-n-2, \ldots$. 
\end{proposition}
We now fix the variable $\vartheta \in \mathbb{R}^s$ and we consider $\Delta_{0,s}$ as an operator with respect to $\xi \in \mathbb{R}^{2n}$. Assuming that the real part of 
$\lambda$ is sufficiently large so that all boundary integrals that appear via partial integration vanish we calculate: 
\begin{align*}
\Big{(}P_{\pm}^{\lambda}, \mathcal{F} \big{[} \Delta_{0,s} \varphi( \cdot , \vartheta) \big{]} \Big{)}
&=\int_{\pm P(\xi) >0}\big{(}\pm P(\xi)\big{)}^{\lambda}  \mathcal{G}_{0,s} \big{[} \mathcal{F} \varphi \big{]}(\xi, \vartheta) \; \dd \xi\\
&=- \int_{\pm P(\xi) >0} \big{(} \pm P(\xi) \big{)}^{\lambda}\left\{ P(\xi) + \frac{|\vartheta|^2}{4}\mathcal{L} \right\} \big{[} \mathcal{F} \varphi \big{]}(\xi, \vartheta) \; \dd \xi.  
\end{align*}
Clearly, the same formula holds true if we replace the distributions $P_{\pm}^{\lambda}$ by $(P\pm i0)^{\lambda}$ above. Now, we assume again that $n=2$ and we rewrite the last 
equation in the form: 
\begin{equation*} 
\Big{(} \big{(}P-i0)^{\lambda}, \mathcal{F} \big{[} \Delta_{0,s} \varphi(\cdot, \vartheta) \big{]} \Big{)}_{|_{\lambda=-1}}=-\int_{\mathbb{R}^4} \big{[} \mathcal{F} \varphi \big{]}(\xi, \vartheta) \; \dd \xi
- \frac{|\vartheta|^2}{4} \Big{(}(P-i0)^{\lambda}, \mathcal{L} \big{[} \mathcal{F} \varphi \big{]}(\cdot, \vartheta) \Big{)}_{|_{\lambda=-1}}. 
\end{equation*}
The following relations are well-known (cf. \cite{Gelfand_Shilov} , p. 258 and formula (4) on p. 277 with $k=0$) and can be applied to the right hand side of the equation: 
\begin{align*}
\Big{(} P_{\pm}^{\lambda-1}, \varphi\Big{)} 
= \frac{\pm1}{4 \lambda(\lambda+1)}  \Big{(} P_{\pm}^{\lambda}, \mathcal{L} \varphi \Big{)} \hspace{2ex} \mbox{\it and} \hspace{2ex} 
\textup{res}_{\lambda=-2}(P+i0)^{\lambda}=- \pi^2 \delta_0,
\end{align*}
where $\delta_0$ means the point evaluation at zero. Therefore we have: 
\begin{equation*}
 \Big{(} (P-i0)^{\lambda}, \mathcal{L} \varphi \Big{)}= 4 \lambda(\lambda+1) \Big{[}\big{(}P_+^{\lambda-1}, \varphi \big{)}+ e^{- \pi i(\lambda-1)} \big{(}P_-^{\lambda-1}, \varphi \big{)} \Big{]}. 
\end{equation*}
Combining these formulas and using Proposition \ref{Propoosition_meromorphic_extension_distribution} gives: 
\begin{equation*}
\Big{(} \big{(} P-i0 \big{)}^{\lambda}, \mathcal{L} \varphi \Big{)}_{|_{\lambda=-1}} =-4 \textup{res}_{\lambda=-2} \Big{(} (P+i0)^{\lambda}, \varphi \Big{)} = (2\pi)^2 \varphi(0). 
\end{equation*}
If we would interpret the action of the distribution $\frac{1}{P}$ in (\ref{candidate_for_a_fundamental_solution_n=2}) as $(P-i0)^{\lambda}_{|_{\lambda=-1}}$, then:
\begin{align*}
\widetilde{K}_{0,s}\big{(} \Delta_{0,s} \varphi \big{)}
&= - \frac{1}{(2\pi)^{2+\frac{s}{2}}} \int_{\mathbb{R}^s} \Big{(} \big{(}P-i0)^{\lambda}, \mathcal{F} \big{[} \Delta_{0,s} \varphi(\cdot, \vartheta) \big{]} \Big{)}_{|_{\lambda=-1}} \; \dd \vartheta\\
&=\frac{1}{(2\pi)^{2+ \frac{s}{2}}} \int_{\mathbb{R}^{4+s}} \big{[} \mathcal{F}\varphi\big{]}(\xi, \vartheta) \; \dd(\xi, \vartheta) + \frac{1}{(2\pi)^{\frac{s}{2}}} \int_{\mathbb{R}^s} \frac{|\vartheta|^2}{4} 
\big{[} \mathcal{F} \varphi \big{]}(0, \vartheta) \; \dd \vartheta\\
&= \varphi(0)+ \frac{1}{4} \big{(} \Delta_{\vartheta} \varphi \big{)}(0). 
\end{align*}
Here we write $\Delta_{\vartheta}$ for the Laplace operator with respect to $\vartheta \in \mathbb{R}^s$. So we have seen that $\widetilde{K}_{s,0}$ fails to be a fundamental solution of 
$\Delta_{0,s}$ but rather solves the equation: 
\begin{equation*}
\Delta_{0,s} \widetilde{K}_{0,s}= \delta_0+ \frac{1}{4} \big{(}\delta_{x=0} \mathcal{F}\big{)} \otimes \big{(}\Delta_{\vartheta} \delta_{z=0}\big{)}. 
\end{equation*}
}
\section{A second form of  fundamental solutions}
\label{Section_A second form of  fundamental solutions}
In the present section we represent the fundamental solutions $K_{0,s}^{\lambda,\mu}$ of the ultra-hyperbolic operator $\Delta_{0,s}$ in Theorem \ref{theorem_family_of_fundamental_solutions} in 
a different form. The formulas which we obtain generalize the expressions derived in \cite{MuellerRicci_2,Tie} in the special case of the Heisenberg Lie algebra, i.e., $s=1$ and for a specific choice of $\lambda$ and $\mu$.  
We will use the notation in Sections \ref{Section_From_Sub-Laplacian_to_UHO}, \ref{section_3} and \ref{Section_non-uniqueness-of-the-fundamental-solution}. For simplicity we first consider 
the fundamental solution $q_{0,s}^{1,0}$ which arises for $(\lambda,\mu)=(1,0)$ and therefore will be denoted by 
$K_{0,s}^{1,0}:=K_{0,s}$. We take $\varphi(x,z) \in \mathcal{S}(\mathbb{R}^{2n+s})$ and interchange the order of integration in (\ref{Definition_K_0_s}):
\begin{align}
\label{eq:K_0_s_+_Section_7}
K_{0,s}(\varphi):
&= \frac{1}{(2\pi)^{n+s/2}}\int_0^{\infty} \int_{\mathbb{R}^{2n+s}} G_0(\xi, \vartheta, t) \big{[} \mathcal{F} \varphi\big{]}(\xi, \vartheta) \; \dd(\xi, \vartheta) \; \dd t \\
&= \frac{1}{(2\pi)^{n+s/2}}\int_0^{\infty} \lim_{\varepsilon \rightarrow 0}\int_{\mathbb{R}^{2n+s}}  G_{\varepsilon}(\xi, \vartheta, t)
 \big{[} \mathcal{F}\varphi\big{]}(\xi, \vartheta) \; \dd(\xi, \vartheta) \; \dd t, \notag
\end{align}
where with $\varepsilon \geq 0$ we define 
\begin{equation*}
G_{\varepsilon}(\xi, \vartheta,t) :=\frac{i}{|\vartheta| \cosh^n t} \exp \left\{ i \frac{\tanh t}{|\vartheta|}P(\xi) - \frac{\varepsilon |\vartheta|}{4}\coth t \right\}. 
\end{equation*}
According to Lemma \ref{Lemma_Fourier_transformLexpontion_quadratic_form} and with the notation in (\ref{W_r_=_0}) we have:
\begin{align}\label{eq:integrarion_G_varepsilon_Fourier_transform}
    &\int_{\mathbb{R}^{2n+s}} G_{\varepsilon}(\xi, \vartheta, t) \big{[} \mathcal{F} \varphi\big{]}(\xi, \vartheta) \; \dd(\xi, \vartheta) =\\
          =&\int_{\mathbb{R}^{2n+s}}\Big{[}\mathcal{F}_{\xi \rightarrow x} G_{\varepsilon} \Big{]} (x, \vartheta, t)
   \Big{[} \mathcal{F}_{z \rightarrow \vartheta} \varphi \Big{]} (x, \vartheta) \; \dd(x, \vartheta)\notag\\
         =&\frac{W(2t)}{(2t)^n} \int_{\mathbb{R}^{2n+s}} |\vartheta|^{n-1}\exp \left\{ -\frac{i}{4} |\vartheta|   \big{[} P(x)-i \varepsilon\big{]} \coth t \right\}
\Big{[} \mathcal{F}_{z \rightarrow \vartheta} \varphi \Big{]} (x, \vartheta) \; \dd(x, \vartheta)=(*). \notag
\end{align}
\par 
The distribution $K_{0,s}^{0,1}$ corresponding to the kernel $q_{0,s}^{0,1}$ in (\ref{Fundamental_solution_general_form_lambda_mu}) can be expressed in the same 
way by replacing $G_{\varepsilon}$ with its complex conjugate $\overline{G_{\varepsilon}}$. More generally, from these we can derive integral expression for the general case 
$q_{0,s}^{\lambda, \mu}$ where $\lambda+\mu\equiv 1: \mathbb{R}^s \rightarrow \mathbb{C}$:  
\begin{multline}\label{Integral_expression_after_FT_K_0_s_lambda_mu}
K_{0,s}^{\lambda, \mu} (\varphi)= \frac{1}{(2\pi)^{n+ \frac{s}{2}}} \int_0^{\infty} \lim_{\varepsilon \rightarrow 0} \int_{\mathbb{R}^{2n+s}} \big{[} \mathcal{F}_{z \rightarrow \vartheta} \varphi \big{]} (x, \vartheta) \times \\
\times \Big{(} \lambda \big{[} \mathcal{F}_{\xi \rightarrow x} G_{\varepsilon} \big{]} + \mu \big{[} \mathcal{F}_{\xi \rightarrow x} \overline{G_{\varepsilon}} \big{]} \Big{)}(x, \vartheta, t)
 \;  (x, \vartheta) \; \dd t. 
\end{multline}
\begin{example}
\label{Heisenberg_group_case_exampele_Mueler_Ricci_Tie}
In case of the Heisenberg group $G_{0,1}$ of dimension $2n+1$ in  Example \ref{Example_Heisenberg_Lie_algebra} and via a specific choice of $\lambda$ and $\mu$ 
we recover from the last formula an expression of a fundamental solution of the ultra-hyperbolic operator $\Delta_{0,1}$ which previously has been presented in the work by D. M\"{u}ller and F. Ricci in 
\cite{MuellerRicci}, (see also \cite{Tie}, p. 1297). 
\vspace{1ex}\par 
Let $s=1$ and choose $\lambda: \mathbb{R} \rightarrow [0,1]$ to be the characteristic function of the non-negative half-line:
\begin{equation*}
\lambda(\vartheta):= 
\begin{cases}
1, & \textup{\it if} \: \: \vartheta \geq 0, \\
0, & \textup{\it if} \: \: \vartheta<0.  
\end{cases}\hspace{4ex} \mbox{\it and put} \hspace{4ex} \mu(\vartheta):= 1- \lambda(\vartheta). 
\end{equation*}
Then for all $\vartheta \in \mathbb{R}\setminus \{0\}$ we have: 
\begin{equation*}
\lambda G_{\varepsilon}+ \mu \overline{G_{\varepsilon}}
= \frac{i}{\vartheta \cosh^nt} \exp \left\{ i \frac{\tanh t}{\vartheta} P(\xi) - \frac{\varepsilon |\vartheta|}{4} \coth t \right\}. 
\end{equation*}
Applying Lemma \ref{Lemma_Fourier_transformLexpontion_quadratic_form} gives: 
\begin{equation*}
\lambda \big{[} \mathcal{F}_{\xi \rightarrow x}G_{\varepsilon}\big{]}+ \mu \big{[} \mathcal{F}_{\xi \rightarrow x} \overline{G_{\varepsilon}}\big{]}
=
\frac{i |\vartheta|^n}{\vartheta \sinh^nt} \exp \left\{ - \frac{i}{4} \vartheta \big{[} P(\xi) -i \varepsilon \textup{sgn}(\vartheta) \big{]} \coth t \right\}. 
\end{equation*}
As usual, $\textup{sgn}(\vartheta)$ denotes the sign-function. Therefore, in this example we find: 
\begin{multline*}
\lim_{\varepsilon \rightarrow 0} \int_{\mathbb{R}^{2n+1}} \big{[}\mathcal{F}_{z \rightarrow \vartheta} \varphi \big{]}(x, \vartheta) 
\Big{(} \lambda \big{[} \mathcal{F}_{\xi \rightarrow x}G_{\varepsilon}\big{]}+ \mu \big{[} \mathcal{F}_{\xi \rightarrow x} \overline{G_{\varepsilon}}\big{]}\Big{)}(x, \vartheta,t) \; \dd (x, \vartheta)=\\
= \frac{i}{2^n \sinh ^n t} \int_{\mathbb{R}^{2n+1}} (\textup{sgn } \vartheta)^n \vartheta^{n-1}\big{[}\mathcal{F}_{z \rightarrow \vartheta} \varphi \big{]}(x, \vartheta) 
\exp \left\{ - \frac{i}{4} \vartheta P(x) \coth t \right\} \; \dd (x, \vartheta)=(+). 
\end{multline*}
If $n$ is even we can calculate the integral on the right more explicitly by the inversion of the Fourier transform: 
\begin{equation*}
(+)= \frac{\sqrt{2 \pi}}{(2i \sinh t)^n} \int_{\mathbb{R}^{2n}}\frac{\partial^{n-1} \varphi}{\partial z^{n-1}} \Big{(}x, - \frac{P(x)}{4} \coth t \Big{)} \; \dd x. 
\end{equation*}
Inserting the last formula into (\ref{Integral_expression_after_FT_K_0_s_lambda_mu}) gives a fundamental solution in form of an iterated integral: 
\begin{equation}\label{Fundamental_solution_Mueller_Ricci}
K_{0,1}^{\lambda, \mu}(\varphi)=\frac{1}{(4 \pi i)^n} \int_0^{\infty} \frac{1}{\sinh^nt} \frac{\partial^{n-1}\varphi}{\partial z^{n-1}} \left(x,- \frac{P(x)}{4} \coth t \right) \; \dd x \; \dd t. 
\end{equation}
Up to a  sign this is the fundamental solution derived in (62) of \cite{Tie}. 
{Since $\coth t >1$ on $(0,\infty)$ it follows that the distribution (\ref{Fundamental_solution_Mueller_Ricci}) vanishes in 
$\big{\{} (x,z)\in \mathbb{R}^{2n+s} \: : \:4 |z| < |P(x)|\big{\}}.$ }
\end{example}
In the case of the center dimension $s>1$ and $n>1$ we may pass to polar coordinates in the $\vartheta$-integration of $(*)$. For the rest of the section we only consider the 
fundamental solution $K_{0,s}^{1,0}$. However, all calculations can be done in a similar way for  the general fundamental solutions $K_{0,s}^{\lambda, \mu}$ in (\ref{Integral_expression_after_FT_K_0_s_lambda_mu}). With the standard surface measure $\sigma$ on the euclidean sphere 
$\mathbb{S}^{s-1}$ in $\mathbb{R}^s$ and $r>0$ we put: 
\begin{equation*}
\widetilde{\varphi}(x,r)= \int_{\mathbb{S}^{s-1}} \Big{[} \mathcal{F}_{z \rightarrow \vartheta} \varphi \Big{]}(x,r\omega) \; \dd \sigma(\omega). 
\end{equation*}
Then we have 
\begin{equation*}
(*)=\frac{W(2t)}{(2t)^n} \int_{\mathbb{R}^{2n}} \int_0^{\infty} r^{n+s-2} \exp \left\{- \frac{ir}{4}  \big{[} P(x)-i \varepsilon\big{]} \coth t\right\} \widetilde{\varphi} (x,r) \; \dd r \; \dd x. 
\end{equation*}
In the $r$-integral we can perform $(n-1)$  partial integrations without producing boundary terms and using: 
\begin{equation*}
e^{ - \frac{ir}{4} [P(x)-i \varepsilon] \coth t}= 
\frac{(4i)^{n-1}\tanh^{n-1}t}{[P(x)-i \varepsilon]^{n-1}}   \frac{\dd^{n-1}}{\dd r^{n-1}}e^{ - \frac{ir}{4}[P(x)-i \varepsilon] \coth t}.  
\end{equation*}
It follows that 
\begin{equation*}
(*)=\frac{2^{n-2}}{i^{n-1}}  \frac{1}{\sinh t \cosh^{n-1} t} \int_{\mathbb{R}^{2n}} \frac{ \varphi_{t, \varepsilon}(x)}{[P(x)-i \varepsilon]^{n-1}} \; \dd x, 
\end{equation*}
where the function $\varphi_{t, \varepsilon}(x)$ is given by: 
\begin{equation}\label{Phi_rho_varepsilion}
\varphi_{t, \varepsilon}(x) = \int_0^{\infty} \frac{\dd^{n-1}}{\dd r^{n-1}} \Big{[} r^{n+s-2} \widetilde{\varphi}(x,r) \Big{]} 
e^{ - \frac{ir}{4} [P(x)-i \varepsilon] \coth t}\; \dd r \in \mathcal{S}(\mathbb{R}^{2n}). 
\end{equation}
Summarizing the calculation we have shown: 
\begin{lemma}\label{Lemma_second_form_fundamental_solution_limit}
Let $\varphi =\varphi(x,z) \in \mathcal{S}(\mathbb{R}^{2n+s})$,  then: 
\begin{equation}\label{limit_varepsilon_delta_zero}
 K^{1,0}_{0,s} \big{(} \varphi \big{)} =\left(\frac{2}{i}\right)^{n-2} \frac{1}{(2\pi)^{n+ \frac{s}{2}}}
 \int_0^{\infty} \frac{1}{\sinh t  \cosh^{n-1} t}  \lim_{\varepsilon \rightarrow 0} 
 \int_{\mathbb{R}^{2n}} \frac{\varphi_{t, \varepsilon}(x)}{[P(x)-i \varepsilon]^{n-1}} \; \dd x \; \dd t. 
\end{equation}
\end{lemma}
In the next step we wish to calculate the limit in the $t$-integrand of  (\ref{limit_varepsilon_delta_zero}). We fix a Schwartz function $\psi \in \mathcal{S}(\mathbb{R}^{2n})$ and 
study the existence of the limit 
\begin{equation*}
D(\psi):=\lim_{\varepsilon \rightarrow 0} \int_{\mathbb{R}^{2n}} \frac{\psi(x)}{[P(x)-i \varepsilon]^{n-1}} \; \dd x
=i^{n-1}\lim_{\varepsilon \rightarrow 0} \int_{\mathbb{R}^{2n}} \frac{\psi(x)}{[iP(x)+\varepsilon]^{n-1}} \; \dd x. 
\end{equation*}
Let $\beta \in \mathbb{C}$ with $\textup{Re}(\beta)>0$ and recall the well-known integral representation of the Gamma function: 
\begin{equation}\label{Gamma_function_integral_expression}
\frac{\Gamma(s)}{\beta^s} = \int_0^{\infty} t^{s-1} e^{- \beta t } \; \dd t. 
\end{equation}
It follows that 
\begin{equation*}\label{Integral_P(x)-ivarepsilon_1}
D(\psi) = \frac{i^{n-1}}{\Gamma(n-1)}\lim_{\varepsilon \rightarrow 0}  \int_0^{\infty} t^{n-2} e^{-\varepsilon t} \int_{\mathbb{R}^{2n}} \psi(x) e^{-iP(x)t} \; \dd x \; \dd t. 
\end{equation*}
\par 
We decompose the outer integral into the integrals $\int_0^1$ and $\int_1^{\infty}$. The first  $\varepsilon$-dependent family of integrals will be denoted $I^{\varepsilon}_{1}(\psi)$ and the second by 
$I^{\varepsilon}_{2}(\psi)$. Clearly, $I^{\varepsilon}_{1}(\psi)$ converges as $\varepsilon \rightarrow 0$ to: 
\begin{equation*}
I_{1}(\psi):=\frac{i^{n-1}}{\Gamma(n-1)} \int_0^1 t^{n-2} \int_{\mathbb{R}^{2n}} \psi(x) e^{-iP(x)t} \; \dd x\;  \dd t. 
\end{equation*}
Moreover,  independently of $\varepsilon >0$ and with the $L^1$-norm on $L^1(\mathbb{R}^{2n})$ we have the estimate: 
\begin{equation*}
\label{Estimate_I_1_k_Psi}
\big{|} I^{\varepsilon}_{1}(\psi)\big{|} \leq \frac{1}{(n-1)!}  \| \psi \|_{L^1}.
\end{equation*}
In the next step we consider the integrals over $[1, \infty)$ and we perform a Fourier transform in the $x$-integral. Using Lemma \ref{Lemma_Fourier_transformLexpontion_quadratic_form} we find: 
\begin{equation*}
 \mathcal{F} \Big{[} e^{-iP(x)t} \Big{]}(\xi)= (2t)^{-n} e^{\frac{i}{4t} P(\xi) }.
\end{equation*}
Therefore: 
\begin{equation}\label{Chapter_7_Fourier_transform_integral}
\int_{\mathbb{R}^{2n}}\psi(x) e^{-iP(x) t} dx
=\frac{1}{(2t)^n} \int_{\mathbb{R}^{2n}} \big{[} \mathcal{F}^{-1} \psi \big{]}(\xi)  e^{\frac{i}{4t} P(\xi)} \; \dd \xi. 
\end{equation}
According to the last expression we find: 
\begin{align*}
I^{\varepsilon}_{2}(\psi)
=& \frac{i^{n-1}}{2^n\Gamma(n-1)} \int_1^{\infty} \frac{e^{-\varepsilon t} }{t^2} \int_{\mathbb{R}^{2n}} \big{[}\mathcal{F}^{-1} \psi \big{]}(\xi) e^{\frac{i}{4t} P(\xi)} \; \dd \xi \; \dd t. 
\end{align*}
The iterated integrals exist in the case $\varepsilon =0$ and by the dominated convergence theorem we have
\begin{equation*}
\lim_{\varepsilon \rightarrow 0}I^{\varepsilon}_{2}(\psi)=\frac{i^{n-1}}{2^n\Gamma(n-1)} \int_1^{\infty} \frac{1}{t^2} \int_{\mathbb{R}^{2n}} \big{[}\mathcal{F}^{-1} \psi \big{]}(\xi) 
 e^{\frac{i}{4t} P(\xi)} \; \dd \xi \; \dd t. 
\end{equation*}
Moreover,  independent of $\varepsilon >0$ one obtains the estimate:
\begin{equation*}
\big{|} I_{2}^{\varepsilon}(\psi) \big{|} \leq \frac{1}{2^n\Gamma(n-1)} \big{\|} \mathcal{F}^{-1} \psi \big{\|}_{L^1}. 
\end{equation*}
We summarize the above observations: 
\begin{proposition}\label{Proposition_Distribution_P(x)-ivarepsilon}
Let $\psi \in \mathcal{S}(\mathbb{R}^{2n})$. Then the limit 
\begin{equation*}
\lim_{\varepsilon \rightarrow 0} \int_{\mathbb{R}^{2n}} \frac{\psi(x)}{[P(x)-i \varepsilon]^{n-1}} \; \dd x =: \frac{1}{P^{n-1}} \big{[} \psi \big{]} 
\end{equation*}
exists  and there is a constant $C>0$ independent of $\psi$ and $\varepsilon$ such that: 
\begin{equation}\label{GL_inequality_distribution_P_Fourier_transform}
\left| \int_{\mathbb{R}^{2n}} \frac{\psi(x)}{[P(x)-i \varepsilon]^{n-1}} \; \dd x \right| \leq C\Big{(} \big{\|} \psi\|_{L^1}+ \big{\|} \mathcal{F}^{-1}\psi \big{\|}_{L^1} \Big{)}. 
\end{equation}
\end{proposition} 
\begin{remark}
In \cite{Gelfand_Shilov} various distributions associated to the quadratic form $P(x)$ have been defined. We remark that $1/P^{n-1}$ in Proposition \ref{Proposition_Distribution_P(x)-ivarepsilon} 
coincides with the value of $(P+i0)^{\lambda}$ in \cite[Chapter III, Section 2.4]{Gelfand_Shilov} at $\lambda=-n+1$, cf. Proposition \ref{Propoosition_meromorphic_extension_distribution}. More precisely, it holds: 
\begin{equation}\label{GL_comparison_Distribution_GS_remark}
\frac{1}{P^{n-1}} \big{[} \psi \big{]} = \Big{(} \big{(}P+i0\big{)}^{-n+1},\psi\Big{)} \hspace{3ex} \mbox{\it for all } \hspace{3ex} \psi \in \mathcal{S}(\mathbb{R}^{2n}). 
\end{equation}
Equation (\ref{GL_comparison_Distribution_GS_remark}) is proven in Proposition \ref{Appendix_proposition_1} of the Appendix. In the case $n=2$ the left hand 
side of (\ref{GL_comparison_Distribution_GS_remark}) has been represented in another form in \cite[Formula (5.6), p. 331]{MuellerRicci}. 
\end{remark}
In order to modify inequality (\ref{GL_inequality_distribution_P_Fourier_transform}) we perform a standard estimate: 
\begin{align*}
\big{\|} \mathcal{F}^{-1}\psi \big{\|}_{L^1}=\frac{1}{(2\pi)^n} \Big{\{} \int_{|x| \leq 1} +\int_{|x|>1} \Big{\}} \Big{|}\int_{\mathbb{R}^{2n}} \psi(y) e^{ix \cdot y} \; \dd y\Big{|} \; \dd x. 
\end{align*}
We can choose a constant $c>0$ such that the first  integral can be estimated by: 
\begin{equation*}
 \int_{|x| \leq 1} \int_{\mathbb{R}^{2n}} |\psi(y)| \; \dd y \; \dd x \leq c \| \psi \|_{L^1}. 
\end{equation*}
Let $\Delta= \sum_{j=1}^{2n} \partial_{x_j}^2$ denote the Laplace operator on $\mathbb{R}^{2n}$ and fix $\ell \in \mathbb{N}$. Then we can use the relation 
\begin{equation*}
\mathcal{F}^{-1} \Delta^{\ell} \psi=(-1)^{\ell} |\cdot |^{2\ell} \mathcal{F}^{-1} \psi 
\end{equation*}
and choose $\ell \in \mathbb{N}$ sufficiently large such that $|\cdot|^{-2\ell}$ becomes an integrable function over the exterior domain $\{ x \in \mathbb{R}^{2n} \: : \: |x|>1\}$ of the unit ball. 
Then we have: 
\begin{align*}
\int_{|x|>1} \big{|} \mathcal{F}^{-1} \psi\big{|}(x) \; \dd x
&=\int_{|x|>1} |x|^{-2\ell} \Big{|} \mathcal{F}^{-1} \Delta^{\ell} \psi(x) \Big{|} \; \dd x\\
&\leq \frac{1}{(2\pi)^n} \int_{|x|>1} \int_{\mathbb{R}^{2n}} |x|^{-2\ell} \big{|} (\Delta^{\ell}\psi) (u) \big{|} \; \dd u\; \dd x \leq C \| \Delta^{\ell} \psi \big{\|}_{L^1}. 
\end{align*}
Hence, with a suitable constant $d>0$ we have for all $\psi \in\mathcal{S}(\mathbb{R}^{2n})$: 
\begin{equation*}
\big{\|} \mathcal{F}^{-1} \psi \big{\|}_{L^1} \leq d \Big{(} \| \psi\|_{L^1} + \big{\|} \Delta^{\ell} \psi \big{\|}_{L^1} \Big{)}. 
\end{equation*}
Therefore the following corollary of Proposition \ref{Proposition_Distribution_P(x)-ivarepsilon} holds: 
\begin{corollary}\label{Proposition_Distribution_P(x)-ivarepsilon_2_version}
There is $\ell \in \mathbb{N}$ and a constant $C_{\ell}>0$ independent of $\psi \in \mathcal{S}(\mathbb{R}^{2n})$ and $\varepsilon >0$ such that 
\begin{equation*}
\left| \int_{\mathbb{R}^{2n}} \frac{\psi(x)}{[P(x)-i \varepsilon]^{n-1}} \; \dd x \right| \leq C_{\ell} \Big{(} \| \psi\|_{L^1} + \big{\|} \Delta^{\ell} \psi \big{\|}_{L^1} \Big{)}. 
\end{equation*}
\end{corollary}\label{Corollary_bound_Fourier_transfrom_Psi}
With the notation in (\ref{Phi_rho_varepsilion}) we estimate 
\begin{align*}
\int_{\mathbb{R}^{2n}} \frac{\varphi_{t, \varepsilon}(x)}{[P(x)-i \varepsilon]^{n-1}} \; \dd x = \int_{\mathbb{R}^{2n}} \frac{\varphi_{t,0}(x)}{[P(x)-i \varepsilon]^{n-1}} \; \dd x
+  \int_{\mathbb{R}^{2n}} \frac{\varphi_{t, \varepsilon}(x)- \varphi_{t,0}(x)}{[P(x)-i \varepsilon]^{n-1}} \; \dd x. 
\end{align*}
According to Proposition \ref{Proposition_Distribution_P(x)-ivarepsilon} the first integral converges for all $t >0$ as $\varepsilon \rightarrow 0$ with limit:
\begin{equation*}
\lim_{\varepsilon \rightarrow 0} \int_{\mathbb{R}^{2n}} \frac{\varphi_{t,0}(x)}{[P(x)-i \varepsilon]^{n-1}} \; \dd x=\frac{1}{P^{n-1}} \big{[} \varphi_{t,0}\big{]}. 
\end{equation*}
Since the $L^1$-norms: 
\begin{equation*}
\big{\|}\varphi_{t, \varepsilon}- \varphi_{t,0} \big{\|}_{L^1} \hspace{3ex} \mbox{\it and} \hspace{3ex} \big{\|} \Delta^{\ell} \big{(} \varphi_{t, \varepsilon}- \varphi_{t,0} \big{)} \big{\|}_{L^1}
\end{equation*}
tend to zero as $\varepsilon \rightarrow 0$, it follows again from Corollary \ref{Proposition_Distribution_P(x)-ivarepsilon_2_version}  that 
\begin{equation*}
\lim_{\varepsilon \rightarrow 0} \int_{\mathbb{R}^{2n}} \frac{\varphi_{t, \varepsilon}(x)- \varphi_{t,0}(x)}{[P(x)-i \varepsilon]^{n-1}} \: \dd x=0.  
\end{equation*}
Therefore we obtain the following second form of the fundamental solution: 
\begin{theorem}
Let $\varphi = \varphi(x,z) \in \mathcal{S}(\mathbb{R}^{2n+s})$. With the notation in Lemma \ref{Lemma_second_form_fundamental_solution_limit}  we have: 
\begin{equation*}
K^{1,0}_{0,s} (\varphi)=\left(\frac{2}{i}\right)^{n-2} \frac{1}{(2\pi)^{n+ \frac{s}{2}}}
 \int_0^{\infty} \frac{1}{\sinh t \cosh^{n-1} t}    \frac{1}{P^{n-1}} \big{[} \varphi_t \big{]} \; \dd t,
\end{equation*}
where
\begin{equation*}
       \varphi_{t}(x):=\varphi_{t,0}(x)=  \int_0^{\infty} \frac{\dd^{n-1}}{\dd r^{n-1}} \Big{[} r^{n+s-2} \widetilde{\varphi}(x,r) \Big{]} 
                                 e^{- \frac{ir}{4} P(x)\coth t }\;  \dd r\in \mathcal{S}(\mathbb{R}^{2n}). 
\end{equation*}
\end{theorem}
\section{On the singular support of $K^{1,0}_{0,s}$}
\label{Section_Singular_set_K_0_s_+}
In order to obtain some information on the singular support of the distribution $K_{0,s}$ in (\ref{eq:K_0_s_+_Section_7}) (or Theorem \ref{th:fundamental_solution}) we 
move the Fourier transform in the integral (\ref{eq:integrarion_G_varepsilon_Fourier_transform}) from the function $\varphi$ to the integral kernel. More precisely, we need to determine 
the Fourier transforms 
\begin{equation*}
\mathcal{F}_{\vartheta \rightarrow z} \Big{\{} |\vartheta|^{n-1} e^{ - \lambda |\vartheta|} \Big{\}}(z)\in L^2(\mathbb{R}^s), 
\end{equation*}
where $\lambda \in \mathbb{C}$ with $\textup{Re}(\lambda)>0$. Recall the following formula (cf. \cite{Taylor}, p. 219):  
\begin{equation*}
\mathcal{F}_{\vartheta \rightarrow z} \Big{[} e^{- \lambda |\vartheta|}\Big{]}(z)= c_s\frac{\lambda}{(\lambda^2+|z|^2)^{\frac{s+1}{2}}}, \hspace{3ex} \mbox{\it where } 
\hspace{3ex} c_s= \frac{2^{\frac{n}{2}}}{\sqrt{\pi}} \Gamma \Big{(} \frac{s+1}{2} \Big{)}.  
\end{equation*}
Taking derivatives with respect to $\lambda$ under the integral yields:
\begin{align}
\mathcal{F}_{\vartheta \rightarrow z} \Big{[}|\vartheta|^{n-1} e^{- \lambda |\vartheta|}\Big{]}(z)
&=(-1)^{n-1} \frac{\dd^{n-1}}{\dd \lambda^{n-1}} \mathcal{F}_{\vartheta \rightarrow z} \Big{[}e^{- \lambda |\vartheta|} \Big{]}(z)\notag\\
&=(-1)^{n-1}c_s \frac{\dd^{n-1}}{\dd \lambda^{n-1}} \frac{\lambda}{(\lambda^2+|z|^2)^{\frac{s+1}{2}}}\notag\\
&=\sum_{j=0}^{n-1} \frac{Q_j(\lambda)}{(\lambda^2+|z|^2)^{\frac{s+1}{2}+j}},
\label{eq:finite_sum_Q_j}
\end{align}
where $Q_j$ is a polynomial of the variable $\lambda$ of some degree $\alpha_j:=\textup{deg} \: Q_j  \in \mathbb{N}_0$. By induction one verifies:   
\begin{equation}\label{eq:estimate_difference_of_the_degree}
2 \left( \frac{s+1}{2} +j\right) -\alpha_j \geq s+n-1, \hspace{3ex} \mbox{\it where} \hspace{3ex} j=1, \ldots,n-1. 
\end{equation}
In particular, with $\lambda:= \frac{1}{4}[iP(x)+\varepsilon]\coth t$ we find 
\begin{equation}\label{eq:lambda_plus_z}
\lambda^2+|z|^2= \frac{1}{16} \Big{(} -P(x)^2+2iP(x) \varepsilon + \varepsilon^2 \Big{)}\coth^2 t +|z|^2. 
\end{equation}
By using (\ref{eq:integrarion_G_varepsilon_Fourier_transform}) it follows that: 
\begin{align}\label{eq:integral_singularities}
\int_{\mathbb{R}^{2n+s}} G_{\varepsilon}(\xi, \vartheta, t) \big{[} \mathcal{F} \varphi \big{]} (\xi, \vartheta) \; \dd (\xi, \vartheta) 
&=\frac{1}{(2\sinh t)^n}\sum_{j=0}^{n-1}\int_{\textup{supp}(\varphi)} \frac{Q_j(\lambda)\varphi(x,z)}{(\lambda^2+|z|^2)^{\frac{s+1}{2}+j}} \; \dd (x,z). 
\end{align}
Suppose that 
\begin{equation}\label{eq:condition_varphi_support}
\textup{supp}(\varphi) \: \cap \:   \Big{\{} (x,z) \in \mathbb{R}^{2n+s} \: : \: 4|P(x)| \leq |z|\Big{\}} =  \emptyset, 
\end{equation}
and define: 
\begin{align*}
S:= \Big{\{} (x,z) \in \mathbb{R}^{2n+s} \: : \: P(x)=0\Big{\}}. 
\end{align*}
\par
By the condition (\ref{eq:condition_varphi_support}) and because of $\coth t >1$ for all $t >0$ it is clear from (\ref{eq:lambda_plus_z}) that we may put $\varepsilon=0$ on the right hand side 
of (\ref{eq:integral_singularities}) without causing a singularity in the integrand. Since $\textup{supp}(\varphi)$ does not intersect with $S$ we find from (\ref{eq:lambda_plus_z}) 
in the case $\varepsilon=0$:
\begin{align*}
\lambda^2+|z|^2&=- \frac{P(x)^2}{16} \coth^2 t +|z|^2 \in O(t^{2})\\
Q_j(\lambda) &= O(t^{\alpha_j}) \hspace{4ex} \mbox{\it as} \hspace{4ex} t \rightarrow 0. 
\end{align*}
Hence (\ref{eq:finite_sum_Q_j}) and (\ref{eq:estimate_difference_of_the_degree}) imply: 
\begin{equation*}
\int_{\textup{supp}(\varphi)} \frac{Q_j(\lambda)}{(\lambda^2+|z|^2)^{\frac{s+1}{2}+j}} \varphi(x,z) \; \dd (x,z)
\in O(t^{s+n-1}) \hspace{3ex} \mbox{\it as} \hspace{3ex} t \downarrow 0. 
\end{equation*}
Choose $\varepsilon=0$ such that $\lambda=\lambda_0= \frac{i}{4}P(x)\coth t $ and $(x,z) \in \mathbb{R}^{2n+s}$. Put
\begin{equation*}
K(x,z):=\int_0^{\infty}\frac{1}{(2\sinh t)^n}\sum_{j=0}^{n-1}\frac{Q_j(\lambda_0)}{\Big{(}-\frac{P(x)^2}{4} \coth^2 t+|z|^2\Big{)}^{\frac{s+1}{2}+j}} \; \dd t. 
\end{equation*}
Assuming (\ref{eq:condition_varphi_support}) for $\varphi \in \mathcal{S}(\mathbb{R}^{2n+s})$ we have shown that $K(x,z)$ defines a smooth function in a neighborhood of 
$\textup{supp}(\varphi)$.  Moreover, (\ref{eq:integral_singularities}) and (\ref{eq:K_0_s_+_Section_7}) imply that 
\begin{equation*}
K^{1,0}_{0,s}(\varphi)= \frac{1}{(2\pi)^{n+ s/2}} \int_{\mathbb{R}^{2n+s}}K(x,z) \varphi(x,z) \; \dd (x,z). 
\end{equation*}
Summarizing these observations we have shown: 
\begin{theorem}\label{thm:regular_set_K_s_0_+}
The singular support of the fundamental solution $K_{0,s}$ is contained in 
\begin{equation*}
\Big{\{} (x,z) \in \mathbb{R}^{2n+s} \: : \: 4 |P(x)| \leq |z|  \Big{\}}. 
\end{equation*}
\end{theorem}

\begin{example}
In case of the Heisenberg group and the specific choice of the functions $\lambda$ and $\mu$ in  Example \ref{Heisenberg_group_case_exampele_Mueler_Ricci_Tie} we have even seen that, 
cf. p. 1295 in \cite{Tie}
\begin{equation*}
\textup{supp}\big{(}K_{0,1}^{\lambda, \mu} \big{)} \subset \Big{\{} (x,z) \in \mathbb{R}^{2n+s} \: : \: 4 |P(x)| \leq |z|  \Big{\}}.
\end{equation*} 
\end{example}
A close relation between the structure of geodesics tangent to the left invariant distribution spanned by $\{X_1,\ldots,X_{2n}\}$ and the cone described by $P(x)=0$ was 
observed in~\cite{KorolkoMarkina}.
\section{On the fundamental solution of $\Delta_{r,s}$ in the case $r>0$}
\label{s_r_>_0}
In the case $r>0$ is seems difficult to interpret the formal expression of $K_{r,s}(x,y)$ in Section \ref{Section_From_Sub-Laplacian_to_UHO} in a meaningful way as 
a tempered distribution. In fact, in this last section we will show that $\Delta_{r,s}$ with $r>0$ does not even have a fundamental solution in $\mathcal{S}^{\prime}(\mathbb{R}^{2n+r+s})$. In 
Corollary \ref{cor:explicit_form_G_r_s} we calculated the differential operator $\mathcal{G}_{r,s}$ with 
\begin{equation*}
\mathcal{F} \circ \Delta_{r,s} = \mathcal{G}_{r,s} \circ \mathcal{F},
\end{equation*}
where $\mathcal{F}$ denotes the Fourier transform on $\mathcal{S}(\mathbb{R}^{2n+r+s})$. By using the relation (\ref{eq:relation_between_Omega_rho_first_place}) between 
$\Omega(\eta)$ and $\rho(\eta)$ we can rewrite $\mathcal{G}_{r,s}$ as:
\begin{equation}\label{eq:expression_L_reminder}
\mathcal{G}_{r,s}\varphi=- P(\xi) \varphi + \frac{|\eta_+|^2-|\eta_-|^2}{4} \mathcal{L}\varphi -2i \Big{\langle} \Omega(\eta) \tau \xi,  \nabla_{\xi} \varphi \Big{\rangle}, 
\end{equation}
where $\mathcal{L}$ was the ultra-hyperbolic operator on $\mathbb{R}^{2n}$ defined in (\ref{Definition_classical_UHO_introduction}). 
\vspace{1ex}\par 
Let us fix $\eta \in \mathbb{R}^{r+s}$. With $\varphi \in \mathcal{S}(\mathbb{R}^{2n})$ we introduce the following two operators:
\begin{equation}\label{Definition_A_vartheta}
A_{\eta}\varphi:= -P(\xi) \varphi + \frac{|\eta_+|^2-|\eta_-|^2}{4} \mathcal{L}\varphi,   \qquad
B_{\eta}\varphi:=- 2i \Big{\langle}\Omega(\eta) \tau \xi, \nabla_{\xi} \varphi \Big{\rangle}. 
\end{equation}
According to (\ref{eq:expression_L_reminder}) the operator $\mathcal{G}_{r,s}$ decomposes as: 
$
\mathcal{G}_{r,s}= A_{\eta}+ B_{\eta}$.
Moreover,  $A_{\eta}$ and $B_{\eta}$ are formally self-adjoint operators on $L^2(\mathbb{R}^{2n})$. 
\vspace{1ex}\par 
Let $\eta \in \mathbb{R}^{r+s}$ and put $c_{\eta}=\frac{1}{|\eta|_{r,s}}:= \big||\eta_+|^2-|\eta_-|^2\big|^{-\frac{1}{2}}$. Consider the positive valued function: 
\begin{equation}\label{Definition_varphi_theta}
\varphi_{\eta}(\xi)= \exp \Big{(} -c_{\eta} |\xi|^2 \Big{)} \in \mathcal{S}(\mathbb{R}^{2n}). 
\end{equation}
Then $\varphi_{\eta}(0)=1$ and we also have 
\begin{align*}
\mathcal{L} \varphi_{\eta}
&= \sum_{j=1}^n \frac{\partial}{\partial \xi_j} \Big{[} -2 \xi_j c_{\eta} \varphi_{\eta} \Big{]}- \frac{\partial}{\partial \xi_{j+n}} \Big{[}-2 \xi_{j+n} c_{\eta} \varphi_{\eta} \Big{]}\\
&= -2n c_{\eta} \varphi_{\eta}+ 4c_{\eta}^2 P(\xi) \varphi_{\eta} + 2n c_{\eta} \varphi_{\eta}= \frac{4}{\big||\eta_+|^2- |\eta_-|^2\big|} P(\xi) \varphi_{\eta}. 
\end{align*}
This calculation shows that 
\begin{equation*}
A_{\eta}\varphi_{\eta}=
\begin{cases}
0 \quad&\text{if}\quad |\eta_+|^2- |\eta_-|^2>0
\\
-2P(\xi)\quad&\text{if}\quad |\eta_+|^2- |\eta_-|^2<0
\end{cases} 
\end{equation*}
\begin{corollary}
\label{injectivity_A_vartheta}
Let $\eta =(\eta_+, \eta_-) \in \mathbb{R}^{r+s}$ with $|\eta_+|^2- |\eta_-|^2>0$. Then $A_{\eta}$ is not injective as an operator on $\mathcal{S}(\mathbb{R}^{2n})$. 
\end{corollary}
Now we wish to express $B_{\eta}$ in a geometric form. Let $\eta \in \mathbb{R}^{r+s}$ be fixed and consider the one-parameter matrix group: 
\begin{equation*}
\mathbb{R} \ni t \mapsto e^{t \Omega(\eta) \tau}  \in \mathbb{R}^{2n \times 2n}. 
\end{equation*}
By applying the relations in Lemma \ref{lemma:properties_of_Omega} we can calculate the exponentials.  
\begin{lemma}\label{Formula_exponentials_Omega}
Let $t \in \mathbb{R}$ and let $\eta \in \mathbb{R}^{r+s} \cong \mathbb{R}^{r,s}$ be fixed. With the notation 
\begin{equation*}
|\eta|_{r,s}:= \sqrt{\big||\eta_+|^2-|\eta_-|^2\big|}
\end{equation*}
one obtains: 
\begin{equation}\label{eq:period1}
e^{t \Omega(\eta) \tau}= 
\begin{cases}
\cos \left( \frac{t |\eta|_{r,s}}{2} \right) I+ \frac{2}{|\eta|_{r,s}} \sin \left( \frac{t |\eta|_{r,s}}{2} \right) \Omega(\eta) \tau
\quad&\text{if}\quad|\eta_+|^2-|\eta_-|^2>0.
\\
\cosh \left( \frac{t |\eta|_{r,s}}{2} \right) I+ \frac{2}{|\eta|_{r,s}} \sinh \left( \frac{t |\eta|_{r,s}}{2} \right) \Omega(\eta) \tau 
\quad&\text{if}\quad|\eta_+|^2-|\eta_-|^2<0.
\end{cases}
\end{equation}
and
\begin{equation}\label{eq:period2}
e^{t \tau\Omega(\eta) }= 
\begin{cases}
\cos \left( \frac{t |\eta|_{r,s}}{2} \right) I+ \frac{2}{|\eta|_{r,s}} \sin \left( \frac{t |\eta|_{r,s}}{2} \right) \tau\Omega(\eta) 
\quad&\text{if}\quad|\eta_+|^2-|\eta_-|^2>0.
\\
\cosh \left( \frac{t |\eta|_{r,s}}{2} \right) I+ \frac{2}{|\eta|_{r,s}} \sinh \left( \frac{t |\eta|_{r,s}}{2} \right) \tau\Omega(\eta)  
\quad&\text{if}\quad|\eta_+|^2-|\eta_-|^2<0.
\end{cases}
\end{equation}
In particular, if $|\eta_+|^2-|\eta_-|^2>0$, then $e^{t\Omega(\eta) \tau}$ and $e^{t\tau\Omega(\eta)}$ are periodic in $t$ with period $q_{\eta}= \frac{4 \pi}{|\eta|_{r,s}}$. 
\end{lemma}
\begin{proof}
We conclude from Lemma~\ref{lemma:properties_of_Omega}, (1) that for $|\eta_+|^2-|\eta_-|^2>0$
\begin{equation*}
\big{(} \Omega(\eta) \tau \big{)}^{2k}= \left( - \frac{|\eta_+|^2-|\eta_-|^2}{4} \right)^k= (-1)^k \left( \frac{|\eta|_{r,s}}{2} \right)^{2k}, 
\end{equation*}
and therefore the power series expansion of the exponential function gives: 
\begin{align*}
e^{t \Omega(\eta) \tau} 
&= \sum_{k=0}^{\infty} \frac{t^{2k}}{(2k)!} \Big{(} \Omega(\eta) \tau \Big{)}^{2k} +
 \Omega(\eta) \tau \sum_{k=0}^{\infty} \frac{t^{2k+1}}{(2k+1)!} \Big{(} \Omega(\eta) \tau \Big{)}^{2k}\\
 &= \sum_{k=0}^{\infty} \frac{(-1)^k}{(2k)!} \left( \frac{t |\eta|_{r,s}}{2} \right)^{2k} +
 \Omega(\eta) \tau \sum_{k=0}^{\infty} \frac{(-1)^k}{(2k+1)!} \left( \frac{t|\eta|_{r,s}}{2} \right)^{2k+1} \frac{2}{|\eta|_{r,s}}. 
\end{align*}
In the case $|\eta_+|^2-|\eta_-|^2<0$ we use 
$$
\big{(} \Omega(\eta) \tau \big{)}^{2k}=(-1)^k \left( \frac{i|\eta|_{r,s}}{2} \right)^{2k},
\quad
\cos(i\alpha)=\cosh\alpha,\quad
\sin(i\alpha)=i\sinh\alpha,\quad\alpha\in\mathbb R
$$
and finish the proof of~\eqref{eq:period1}. 
A similar calculation shows~\eqref{eq:period2}.
\end{proof}

\begin{lemma}
\label{Relation_exp_omega_tau}
Let $\eta \in \mathbb{R}^{r+s}$. Then for all $t \in \mathbb{R}$ one has 
\begin{equation}\label{relation_expontial_Omega_tau_2}
e^{-t \tau \Omega(\eta)}\tau e^{t \Omega(\eta)\tau} = \tau
\quad\text{and}\quad
e^{-t \Omega(\eta)\tau }\tau e^{t \tau\Omega(\eta)} = \tau. 
\end{equation}
In particular, the level sets of $P(\xi)= \sum_{j=1}^n\xi_j^2-\xi_{j+n}^2$ are invariant under the flows  $e^{t\Omega(\eta) \tau}$ and $e^{t\tau\Omega(\eta)}$, i.e., for all $t \in \mathbb{R}$ and $\xi \in \mathbb{R}^{2n}$:
\begin{equation}\label{GL_invariance_property}
P\circ e^{t \Omega(\eta) \tau} \xi = P(\xi)
\quad\text{and}\quad
P\circ e^{t \tau\Omega(\eta)} \xi = P(\xi). 
\end{equation}
\end{lemma}
\begin{proof}

Let $\eta =(\eta_+, \eta_-) \in \mathbb{R}^{r+s}$ with $|\eta_+|^2-|\eta_-|^2>0$ as above and $t \in \mathbb{R}$. By using Lemma~ \ref{Formula_exponentials_Omega} we calculate the following product: 
\begin{align}\label{al:flow}
e^{-t \tau \Omega(\eta)}\tau e^{t \Omega(\eta)\tau} &
=\left[ \cos \left( t \frac{|\eta|_{r,s}}{2} \right) \tau -\frac{\sin(t\frac{|\eta|_{r,s}}{2})}{\frac{|\eta|_{r,s}}{2}}\Big{(} \tau \Omega (\eta) \tau \Big{)} \right]\nonumber
\\
&\times
\left[ \cos \left( t \frac{|\eta|_{r,s}}{2}\right) I +\frac{\sin(t \frac{|\eta|_{r,s}}{2})}{\frac{|\eta|_{r,s}}{2}} \Big{(}\Omega(\eta)\tau \Big{)} 
\right]
\\
&= \cos^2(t\frac{|\eta|_{r,s}}{2}) \tau - \frac{\sin^2(t\frac{|\eta|_{r,s}}{2})}{\frac{|\eta|^2_{r,s}}{4}} \Big{(} \tau \Omega(\eta) \Big{)}^2 \tau. \nonumber
\end{align}
We know that 
$
\Big{(} \tau \Omega(\eta) \Big{)}^2 =- \frac{|\eta_+|^2- |\eta_-|^2}{4} I= - \frac{|\eta|^2_{r,s}}{4} I$  from Lemma~\ref{lemma:properties_of_Omega}, (1) . Inserting it to~\eqref{al:flow} we obtain the necessary equality.

If $|\eta_+|^2-|\eta_-|^2<0$ then we calculate
\begin{equation}\label{eq:flow}
e^{-t \tau \Omega(\eta)}\tau e^{t \Omega(\eta)\tau} =
\cosh^2(t\frac{|\eta|_{r,s}}{2}) \tau - \frac{\sinh^2(t\frac{|\eta|_{r,s}}{2})}{\frac{|\eta|^2_{r,s}}{4}} \Big{(} \tau \Omega(\eta) \Big{)}^2 \tau.
\end{equation}
Substituting $\Big{(} \tau \Omega(\eta) \Big{)}^2 =- \frac{|\eta_+|^2- |\eta_-|^2}{4} I=  \frac{|\eta|^2_{r,s}}{4} I$ into~\eqref{eq:flow}, we get the assertion. 
The second equality in~\eqref{relation_expontial_Omega_tau_2} follows analogously.

To show the second statement we write $P(\xi)= \langle \tau \xi, \xi \rangle$ and calculate
\begin{align*}
P\circ e^{t \Omega(\eta) \tau} \xi 
&=\Big{\langle} \tau e^{t \Omega(\eta) \tau } \xi, e^{t \Omega(\eta) \tau } \xi \Big{\rangle}\\
&= \Big{\langle} e^{-t \tau \Omega(\eta)} \tau e^{t \Omega(\eta) \tau }  \xi , \xi \Big{\rangle}= \big{\langle} \tau \xi, \xi  \big{\rangle}=P(\xi)
\end{align*}
for all $\xi \in \mathbb{R}^{2n}$. 
\end{proof}

\begin{lemma}\label{Corollary_geometric_form_B_theta}
Let $\eta \in \mathbb{R}^{r+s} \cong \mathbb{R}^{r,s}$ be fixed. Then the operator $B_{\eta}$ can be expressed in the form 
\begin{equation*}
B_{\eta} \varphi= -2i  \left.\frac{\mathsf{d}}{\mathsf{d}t}\right|_{t=0}  \varphi \left( e^{t \Omega(\eta) \tau} \xi \right) \hspace{4ex} \mbox{\it where} \hspace{4ex} \varphi \in \mathcal{S}(\mathbb{R}^{2n}). 
\end{equation*}
\end{lemma}
\begin{proof}
Let $\varphi \in \mathcal{S}(\mathbb{R}^{2n})$ and $\xi \in \mathbb{R}^{2n}$. Then
\begin{align*}
 \left.\frac{\mathsf{d}}{\mathsf{d}t}\right|_{t=0} \varphi \left( e^{t \Omega(\eta) \tau} \xi \right)
&= [\nabla_{\xi} \varphi]  \cdot  \left.\frac{\mathsf{d}}{\mathsf{d}t}\right|_{t=0} e^{t \Omega(\eta) \tau} \xi \\
&=\big{(} \nabla_{\xi} \varphi \big{)} \Omega(\eta) \tau \xi= \Big{\langle} \Omega(\eta) \tau \xi, \nabla_{\xi} \varphi  \Big{\rangle} = \frac{i}{2}B_{\eta}(\varphi). 
\end{align*}
\end{proof}

For each $t \in \mathbb{R}$, fixed $\eta \in \mathbb{R}^{r+s}$ and $\varphi \in \mathcal{S}(\mathbb{R}^{2n})$ we define two composition operators $C_{\eta,t}^{(j)}$, $j=1,2$, 
on $\mathcal{S}(\mathbb{R}^{2n})$ by: 
$$
C^{(1)}_{\eta,t} \varphi
= \varphi \left(e^{t \tau \Omega(\eta)} \xi\right),
\qquad
C^{(2)}_{\eta,t} \varphi
= \varphi \left(e^{t  \Omega(\eta)\tau} \xi\right). 
$$

\begin{lemma}
\label{lemma_commuting_composition_P}
The operator of multiplication by $P(\xi)$ on $\mathcal{S}(\mathbb{R}^n)$ commutes with both composition operators $C_{\eta, t}^{(j)}$, $j=1,2$. Moreover, 
$C_{\eta, t}^{(2)}$ and $B_{\eta}$ commute.
\end{lemma}
\begin{proof}
We only treat  $C^{(1)}_{\eta,t}$ and use the invariance property (\ref{GL_invariance_property}) in Lemma \ref{Relation_exp_omega_tau}: 
\begin{align*} 
P C_{\eta,t}^{(1)} \varphi (\xi)
&=P(\xi)  \varphi\left( e^{t \tau \Omega(\eta)}\xi \right)\\
&= P \left( e^{t \tau \Omega(\eta)}\xi \right)  \varphi\left( e^{t \tau \Omega(\eta)}\xi \right)= C_{\eta,t}^{(1)} \big{(}P \varphi \big{)} (\xi). 
\end{align*}
The second statement follows by combining Lemma \ref{Relation_exp_omega_tau} and Lemma~\ref{Corollary_geometric_form_B_theta}. 
\end{proof}
As a consequence of Lemma \ref{lemma_commuting_composition_P} we also have: 
\begin{lemma}
\label{lemma_commuting_composition_L}
Let $\mathcal{F}$ denote the Fourier transform on $\mathcal{S}(\mathbb{R}^{2n})$. Then we have for all $t \in \mathbb{R}$ and $\eta \in \mathbb{R}^{r,s}$: 
\begin{equation}\label{commutation_relation_Fourier_transform}
C^{(2)}_{\eta,t}  \circ \mathcal{F}= \mathcal{F} \circ C^{(1)}_{\eta,t} \hspace{3ex} \mbox{\it and} \hspace{3ex} C^{(1)}_{\eta,t}  \circ \mathcal{F} = \mathcal{F}\circ C^{(2)}_{\eta,t}
\end{equation}
In particular, the ulta-hyperbolic operator $\mathcal{L}$ commutes with $C_{\eta, t}^{(j)}$ for $j=1,2$ and by Lemma \ref{lemma_commuting_composition_P} 
\begin{equation*}
\big{[} A_{\eta}, C_{\eta,t}^{(j)} \big{]}=0 \hspace{4ex} j=1,2. 
\end{equation*}
\end{lemma}
\begin{proof} 
First we calculate the commutation relations (\ref{commutation_relation_Fourier_transform}).  It is sufficient to prove the first formula. 
Let $\varphi \in \mathcal{S}(\mathbb{R}^{2n})$, then: 
\begin{align*}
\mathcal{F} \circ C_{\eta,t}^{(1)} \varphi (\xi)
&=\frac{1}{(2\pi)^n} \int_{\mathbb{R}^{2n}} \varphi \left( e^{t \tau \Omega(\eta)} x\right) e^{-i x \cdot \xi} \; \dd x=(*).
\end{align*}
According to Lemma \ref{lemma:properties_of_Omega}, (2) it follows: 
\begin{equation*}
\left| \det  e^{-t \tau \Omega(\eta)}\right|= \left| \det  e^{-t  \Omega(\eta)\tau}\right| =1
\end{equation*}
and the transformation rule for the integral implies: 
\begin{align*}
(*)&=\frac{1}{(2\pi)^n} \int_{\mathbb{R}^{2n}} \varphi(x)\exp \left\{ 
-i \Big{(} e^{-t \tau \Omega(\eta)} x \Big{)} \cdot \xi  
\right\} \; \dd x\\
&=\frac{1}{(2\pi)^n} \int_{\mathbb{R}^{2n}} \varphi(x)\exp \left\{ 
-i x \cdot \Big{(}e^{t \Omega(\eta) \tau} \xi \Big{)}
\right\}\; \dd x=C_{\eta,t}^{(2)} \circ \mathcal{F} \varphi(\xi). 
\end{align*}
Using (\ref{commutation_relation_Fourier_transform}) we can prove the second statement and we only treat the case $j=1$. Considered as operators on $\mathcal{S}(\mathbb{R}^{2n})$ we have
\begin{equation*}
\mathcal{F} \circ P=- \mathcal{F} \circ \mathcal{L} \hspace{4ex} \mbox{\it and} \hspace{4ex} \mathcal{F} \circ \mathcal{L}= -P \circ \mathcal{F}
\end{equation*}
and from (\ref{commutation_relation_Fourier_transform}) and Lemma \ref{lemma_commuting_composition_P} it follows: 
\begin{align*}
\mathcal{F} \circ \mathcal{L} \circ C_{\eta, t}^{(1)} &= - P \circ \mathcal{F} \circ C_{\eta, t}^{(1)} \\
&=-P \circ C_{\eta, t}^{(2)} \circ \mathcal{F}\\
&=-C_{\eta,t}^{(2)} \circ P \circ \mathcal{F}\\
&=C_{\eta,t}^{(2)} \circ \mathcal{F} \circ \mathcal{L}  = \mathcal{F} \circ C_{\eta, t}^{(1)} \circ \mathcal{L}. 
\end{align*}
\par 
Since $\mathcal{F}$ is bijective on $\mathcal{S}(\mathbb{R}^{2n})$ it follows that $\mathcal{L} \circ C_{\eta, t}^{(1)}=C_{\eta, t}^{(1)} \circ \mathcal{L}$. The case $j=2$ can be treated similarly. 
\end{proof}
\begin{corollary}
\label{Corollary_commutation_mathcal_L_composition}
Let $\eta \in \mathbb{R}^{r,s}$ be fixed. Then $B_{\eta}$ commutes with $P$ and $\mathcal{L}$. In particular, $B_{\eta}$ commutes with $A_{\eta}$ 
in (\ref{Definition_A_vartheta}). 
\end{corollary}
\begin{proof}
Since $A_{\eta}$ is a linear combination of $P$ and $\mathcal{L}$ it is sufficient to prove the first statement. Lemma~\ref{Corollary_geometric_form_B_theta} and Lemma \ref{lemma_commuting_composition_P} give with $\varphi \in \mathcal{S}(\mathbb{R}^{2n})$: 
\begin{align*}
B_{\eta} \circ P  \varphi= -2i  \left.\frac{\mathsf{d}}{\mathsf{d}t}\right|_{t=0} C_{\eta,t}^{(2)} \Big{(} P \varphi \Big{)}
=-2i P \left.\frac{\mathsf{d}}{\mathsf{d}t}\right|_{t=0} C_{\eta,t}^{(2)} \big{(} \varphi \big{)}=P\circ B_{\eta} \varphi. 
\end{align*}
Therefore we have $[B_{\eta}, P]=0$. Replacing $P$ with $\mathcal{L}$ in the above calculation and applying Lemma \ref{lemma_commuting_composition_L}  shows 
that also the commutator $[B_{\eta}, \mathcal{L}]$ vanishes. 
\end{proof}

Let now $\varphi \in \mathcal{S}(\mathbb{R}^{2n})$ and fix $\eta \in \mathbb{R}^{r,s}$ with $|\eta_+|^2-|\eta_-|^2>0$. Consider the operator 
\begin{equation*}
D_{\eta}: \mathcal{S}(\mathbb{R}^{2n}) \rightarrow \mathcal{S}(\mathbb{R}^{2n}): \big{[}D_{\eta} \varphi\big{]}(\xi):= \int_0^{q_{\eta}} \big{[}C_{\eta,t}^{(2)} \varphi \big{]} (\xi) \; \dd t, 
\end{equation*}
where $q_{\eta}= \frac{4 \pi}{|\eta|_{r,s}}$ is the period of $\mathbb{R} \ni t \mapsto C_{\eta,t}^{(2)}$ in Lemma \ref{Formula_exponentials_Omega}. Applying the transformation 
$t \rightarrow \frac{4 \pi}{|\eta|_{r,s}} \rho$  we can also write 
\begin{equation*}
\big{[}D_{\eta} \varphi\big{]}(\xi):= \frac{4 \pi}{|\eta|_{r,s}} \int_0^1 \varphi \left( \exp \left\{ \frac{4 \pi \rho}{|\eta|_{r,s}} \Omega(\eta) \tau \right\} \xi \right) \; \dd \rho. 
\end{equation*}
\begin{lemma}\label{Lemma_D_vartheta_maps_into_kernel_of_B_vartheta}
With the notation above $D_{\eta}$ maps $\mathcal{S}(\mathbb{R}^{2n})$ into the kernel of $B_{\eta}$. 
\end{lemma}
\begin{proof}
Let $\varphi \in \mathcal{S}(\mathbb{R}^{2n})$. Then we have by Corollary \ref{Corollary_geometric_form_B_theta}: 
\begin{align*}
B_{\eta} \circ D_{\eta} \varphi (\xi)
&=-2i  \left.\frac{\mathsf{d}}{\mathsf{d}\rho}\right|_{\rho=0}\int_0^{q_{\eta}} C_{\eta, t}^{(2)} \varphi \left( e^{\rho \Omega(\eta) \tau } \xi \right) \; \dd t\\
&=-2i  \left.\frac{\mathsf{d}}{\mathsf{d}\rho}\right|_{\rho=0}\int_0^{q_{\eta}} \varphi \left( e^{(\rho+t) \Omega(\eta) \tau}\xi \right) \; \dd t.
\end{align*}
Since we integrate a periodic function on $\mathbb{R}$ over a full period we have: 
\begin{equation*}
\int_0^{q_{\eta}} \varphi \left( e^{(\rho+t) \Omega(\eta) \tau}\xi \right) \; \dd t=\int_0^{q_{\eta}} \varphi \left( e^{t\Omega(\eta) \tau}\xi \right) \; \dd t
\end{equation*}
and therefore the integrand does not depend on $\rho$. Hence it follows $B_{\eta} \circ D_{\eta} \varphi \equiv 0$ as it was claimed. 
\end{proof}
\begin{lemma}\label{Lemma_D_vartheta_maps_kernel of_A_vartheta_into_kernel_of_B_vartheta}
Let $\eta \in \mathbb{R}^{r,s}$ be fixed with $|\eta_+|^2-|\eta_-|^2 >0$ and choose $\varphi \in \mathcal{S}(\mathbb{R}^{2n})$ in the kernel of $A_{\eta}$. Then 
\begin{equation*}
A_{\eta} \circ D_{\eta} \varphi =0. 
\end{equation*}
Moreover, $D_{\eta}$ maps $\textup{ker} A_{\eta} \subset \mathcal{S}(\mathbb{R}^{2n})$ into $\textup{ker}A_{\eta} \cap \textup{ker} B_{\eta} \subset \mathcal{S}(\mathbb{R}^{2n})$. 
\end{lemma}
\begin{proof}
According to Lemma \ref{Lemma_D_vartheta_maps_into_kernel_of_B_vartheta} it is sufficient to prove the first statement. Let $\varphi \in \textup{ker} A_{\eta} \subset \mathcal{S}(\mathbb{R}^{2n})$. 
Since $A_{\eta}$ commutes with $C_{\eta,t}^{(2)}$ according to Lemma \ref{lemma_commuting_composition_L} we have: 
\begin{equation*}
A_{\eta} \circ D_{\eta} \varphi= \int_0^{q_{\eta}} \big{[}A_{\eta} \circ C_{\eta,t}^{(2)} \varphi\big{]}(\xi) \; \dd t
= \int_0^{q_{\eta}} C_{\eta,t}^{(2)} \Big{(}\underbrace{ A_{\eta} \varphi}_{=0} \Big{)}(\xi)\; \dd t\equiv 0. 
\end{equation*}
This proves the assertion. 
\end{proof}
In particular, we may choose the function $\varphi_{\eta} \in \textup{ker} A_{\eta}$ in (\ref{Definition_varphi_theta}): 
\begin{equation*}
\varphi_{\eta}(\xi)= \exp \Big{(} -c_{\eta} |\xi|^2 \Big{)} \in \mathcal{S}(\mathbb{R}^{2n}) \hspace{3ex} \mbox{\it where} \hspace{3ex} c_{\eta}=\frac{1}{\sqrt{|\eta_+|^2-|\eta_-|^2}}. 
\end{equation*}
Since $\varphi_{\eta}$ only has positive values, the same holds for the function $D_{\eta} \varphi_{\eta}$ and, in particular, it is non-zero. 
\vspace{1ex}\par 
For the moment we consider the operator $\mathcal{G}_{r,s}=A_{\eta}+B_{\eta}$ in Lemma \ref{Lemma_short_form_of_L} for fixed $\eta \in \mathbb{R}^{r,s}$ with $|\eta_+|^2-|\eta_-|^2>0$ 
as an operator on $\mathcal{S}(\mathbb{R}^{2n})$.  Lemma \ref{Lemma_D_vartheta_maps_kernel of_A_vartheta_into_kernel_of_B_vartheta} implies: 
\begin{equation*}
0 \ne D_{\eta} \varphi_{\eta} \in  \textup{ker} A_{\eta} \cap \textup{ker} B_{\eta}  \subset \textup{ker} \big{(}A_{\eta} +B_{\eta}\big{)}= \textup{ker} \: \mathcal{G}_{r,s}\subset 
 \mathcal{S}(\mathbb{R}^{2n}). 
\end{equation*}
Now we consider $\mathcal{G}_{r,s}$ again as an operator on $\mathcal{S}(\mathbb{R}^{2n+r+s})$ as in Lemma  \ref{Lemma_short_form_of_L}. We assume that $r>0$ such that 
\begin{equation*}
\mathcal{K}= \Big{\{} \eta \in \mathbb{R}^{r+s} \: : \: |\eta_+|^2-|\eta_-|^2 >0 \Big{\}} 
\end{equation*}
is a non-empty open subset in $\mathbb{R}^{r+s}$. Choose a compactly supported cut-off function $0 \ne \omega \in C_0^{\infty}(\mathbb{R}^{r+s})$ with 
\begin{equation*}
\textup{supp}\: \omega \subset \mathcal{K}, 
\end{equation*}
and taking values in the interval $[0,1]$. Consider the function: 
\begin{equation*}
\psi(\xi, \eta):=
\omega(\eta)\big{[} D_{\eta} \varphi_{\eta}\big{]} (\xi) \in \mathcal{S}(\mathbb{R}^{2n+r+s}). 
\end{equation*}
\par 
Then by construction $\psi$ only takes values in $\mathbb{R}_+=(0,\infty)$ and 
\begin{equation*}
\mathcal{G}_{r,s} \psi= \omega(\eta) \Big{(} A_{\eta}+B_{\eta} \Big{)} D_{\eta} \varphi_{\eta} \equiv 0. 
\end{equation*}
Hence we have shown: 
\begin{corollary}\label{Kor_function_in_the_kernel_of_L}
Let $r>0$, then there is a non-trivial, non-negative valued function $\psi \in \mathcal{S}(\mathbb{R}^{2n+r+s})$ in the kernel of the operator $\mathcal{G}_{r,s}$.
\end{corollary}
Combining the previous observations, we can now prove the main result: 
\begin{theorem}\label{theorem_existence_of_fundamental_solution_r>0}
Let $r>0$, then the ultra-hyperbolic operator $\Delta_{r,s}$ does not have a fundamental solution in $\mathcal{S}^{\prime}(\mathbb{R}^{2n+r+s})$, i.e., there is no tempered distribution 
$K_{r,s} \in \mathcal{S}^{\prime}(\mathbb{R}^{2n+r+s})$ such that 
\begin{equation*}
\Delta_{r,s} K_{r,s}= \delta_0. 
\end{equation*}
\end{theorem}
\begin{proof}
Let $r>0$ and assume that the ultra-hyperbolic operator $\Delta_{r,s}$ admits a fundamental solution $K_{r,s} \in \mathcal{S}^{\prime}(\mathbb{R}^{2n+r+s})$. Then 
\begin{equation*}
\Delta_{r,s} K_{r,s}(\psi)= K_{r,s} \Big{(} \Delta_{r,s} \psi \Big{)}= \delta_0 \psi= \psi(0) \hspace{4ex} \mbox{\it for all} \hspace{4ex} \psi \in \mathcal{S}(\mathbb{R}^{2n+r+s}). 
\end{equation*}
Choose a non-trivial, non-negative valued function $\psi \in \mathcal{S}(\mathbb{R}^{2n+r+s})$ in the kernel of $\mathcal{G}_{r,s}$ according to Corollary \ref{Kor_function_in_the_kernel_of_L}. 
With the Fourier transform $\mathcal{F}$ on $\mathcal{S}(\mathbb{R}^{2n+r+s})$ we have the relation 
$$\Delta_{r,s}\circ \mathcal{F}^{-1}= \mathcal{F}^{-1} \circ \mathcal{G}_{r,s}$$ 
Hence it follows: 
\begin{align*}
K_{r,s} \big{(} \mathcal{F}^{-1} \circ \mathcal{G}_{r,s} \psi \big{)}
&= K_{r,s} \big{(} \Delta_{r,s} \circ \mathcal{F}^{-1} \psi \big{)}\\
&= [\mathcal{F}^{-1} \psi \big{]}(0)=   \frac{1}{(2\pi)^{n+ \frac{r+s}{2}}} \int_{\mathbb{R}^{2n+r+s}} \psi(\xi,z) \; \dd \xi \; \dd z >0. 
\end{align*}
On the other hand, since $\mathcal{G}_{r,s}\psi=0$ we have $K_{r,s}(\mathcal{F}^{-1} \circ \mathcal{G}_{r,s} \psi )=0$ which leads to a contradiction. 
\end{proof}
\section{On the local solvability of $\Delta_{r,s}$}
\label{Section_Local_solvability}
Recall that a left-invariant differential operator $L$ on $G_{r,s}$ is called {\it locally solvable} at $x_0 \in G_{r,s}$ if one can find an open neighborhood $U$ of $x_0$ such that 
\begin{equation*}
LC^{\infty}(U) \supset C_0^{\infty}(U). 
\end{equation*}
As usual $C_0^{\infty}(U)$ denotes the space of compactly supported smooth functions on $U$. From the left-invariance of $L$ it follows that the local solvability 
of $L$ at a fixed point $x_0$ is equivalent to the local solvability of $L$ at any point in $G_{r,s}$. Hence we may use the term {\it local solvability} without specifying the point. 
\begin{lemma}\label{Lemma_preparation_local_solvability}
Let $r>0$, then there is a function $\varphi \in \mathcal{S}(\mathbb{R}^{2n+r+s})$ which lies in the kernel of $\Delta_{r,s}$ and fulfills $\varphi(0)=1$. 
\end{lemma}
\begin{proof}
Let $\psi\in \mathcal{S}(\mathbb{R}^{2n+r+s})$ be the function in Corollary \ref{Kor_function_in_the_kernel_of_L} and put $\varphi_0= \mathcal{F}^{-1} \psi$. Then 
\begin{equation*}
c:= \varphi_0(0)= \frac{1}{(2 \pi)^{n+ \frac{r+s}{2}}} \int_{\mathbb{R}^{2n+r+s}} \psi(\xi, z) \; \dd \xi \; \dd z>0.
\end{equation*}
If we put $\varphi= c^{-1} \varphi_0$, then $\varphi(0)=1$ and $$\Delta_{r,s}\varphi= c^{-1} \Delta_{r,s} \circ \mathcal{F}^{-1}\psi=c^{-1} \mathcal{F}^{-1}\circ  \mathcal{G}_{r,s}\psi=0$$
since $\psi$ is in the kernel of $\mathcal{G}_{r,s}$.  
\end{proof}
Recall that $G_{r,s}$ is a {\it homogeneous Lie group}, i.e., there is a family $\{ \delta_{\rho}\}_{\rho >0}$ of dilations which at the same time are automorphisms. In fact, with respect to 
the coordinates defined in Section  \ref{Pseudo_H_type_groups} and $\rho >0$, we put 
\begin{equation*}
\delta_{\rho}: G_{r,s} \cong \mathbb{R}^{2n+r+s} \rightarrow G_{r,s}: \delta_{\rho}(x,z):=\big{(}\rho x, \rho^2 z\big{)}. 
\end{equation*}
Then the product formula in (\ref{GL_product_Lie_group}) shows that 
\begin{align*}
\delta_{\rho}\Big{(} (x,z)\:  *\:  (y,w) \Big{)}
&= \Big{(} \rho x+ \rho y, \rho^2 z+ \rho^2 w + \sum_{k=1}^{r+s}  \big{\langle} \Omega_k^T \rho x, \rho y \big{\rangle} {\bf e}_k  \Big{)}\\
&=\delta_{\rho} (x,z)  \: * \: \delta_{\rho}(y,w). 
\end{align*}
From the explicit form of the vector fields $X_j$ in (\ref{eq:vector_fields_X_j_Z_k}) one immediately checks that $\Delta_{r,s}$ is homogeneous of degree $2$ with respect to $\delta_{\rho}$, 
i.e., $\delta_{\rho}^*\Delta_{r,s}= \rho^2 \Delta_{r,s}$ for all $\rho >0$. 
\vspace{1ex}\par 
In this setting Lemma \ref{Lemma_preparation_local_solvability}  can be used to prove that $\Delta_{r,s}$ is locally solvable if and only if $r=0$. The arguments are based on the non-injectivity of 
the ultra-hyperbolic operator $\Delta_{r,s}$ on Schwartz functions in case of $r>0$, see \cite{CR}. We may also use a more refined criterion on {\it local non-solvability} of homogeneous left-invariant 
differential operators which is due to D. M\"{u}ller  and can be found in \cite{Mueller}:
\begin{theorem}[D. M\"{u}ller, \cite{Mueller}]
\label{theorem_Mueller_non_local_solvability}
Let $L$ be a left-invariant homogeneous differential operator on a homogeneous, simply connected nilpotent Lie group $G$ with transpose $L^{\tau}$. Assume there exists a sequence 
$\{\psi_j\}_{j=1}^{\infty}$ of Schwartz functions on $G$ with \textup{(i)} and \textup{(ii)}:  
\begin{itemize}
\item[\textup{(i)}] $\psi_j(0)=1$ for every $j$, 
\item[\textup{(ii)}] For every continuous semi-norm $\| \cdot \|_{(N)}$ on the Schwartz space $\mathcal{S}(G)$ it holds: 
\begin{equation*}
\lim_{j \rightarrow \infty} \| \psi_j \|_{(N)}  \| L^{\tau} \psi_j \|_{(N)}=0.
\end{equation*}
\end{itemize}
Then $L$ is not locally solvable. 
\end{theorem}
Now we can prove the following extension of Theorem \ref{theorem_existence_of_fundamental_solution_r>0}. 
\begin{theorem}\label{theorem_non_local_solvability_r_>_0}
In the case $r>0$  the ultra-hyperbolic operator $\Delta_{r,s}$ is not locally solvable. In particular, $\Delta_{r,s}$ does not even admit a fundamental solution in the space of Schwartz distributions 
$\mathcal{D}^{\prime}(G_{r,s})$. Moreover, $\Delta_{0,s}$ for $s>0$ is locally solvable and 
\begin{equation*}
\Delta_{0,s} C^{\infty}(\mathbb{R}^{2n+s})=C^{\infty}(\mathbb{R}^{2n+s}). 
\end{equation*}
\end{theorem}
\begin{proof}
Since $\Delta_{r,s}$ coincides with its transpose the local non-solvability follows from Theorem \ref{theorem_Mueller_non_local_solvability}  and Lemma \ref{Lemma_preparation_local_solvability}. 
In fact, we may choose $\{ \psi_j\}_j \subset \mathcal{S}(G_{r,s})$ to be the constant sequence $\psi_j= \varphi\in \mathcal{S}(G_{r,s})$ where $\varphi$ denotes the function in Lemma  \ref{Lemma_preparation_local_solvability}. Then \textup{(i)} and \textup{(ii)} above are fulfilled and the first statement follows from Theorem \ref{theorem_Mueller_non_local_solvability}. 
It is known that the following properties are equivalent (see \cite{Battesti,Mueller}): 
\begin{itemize}
\item[(a)] $\Delta_{r,s}$ is locally solvable, 
\item[({b})] $\Delta_{r,s} C^{\infty}(G_{r,s})= C^{\infty}(G_{r,s})$,
\item[({c})] $\Delta_{r,s}$ has a fundamental solution in $\mathcal{D}^{\prime}(G_{r,s})$. 
\end{itemize}
By the equivalences (a) $\Longleftrightarrow$ ({c}) the second statement follows from the first. 
\end{proof}
\section{Appendix}
In this appendix we link the distribution $1/P^{n-1}$ in Proposition \ref{Proposition_Distribution_P(x)-ivarepsilon} to the value of $ (P+ i0)^{\lambda} $ in (\ref{Definition_P_+_-Gelfand}) at 
$\lambda=-n+1$, cf. Proposition \ref{Propoosition_meromorphic_extension_distribution} and \cite{Gelfand_Shilov}. Assume that $\lambda \in \mathbb{C}$ and let $z \in \mathbb{C}$ be 
in the upper half plane, i.e. $\textup{Im}(z)>0$. We write:
\begin{equation*}
 z^{\lambda}=\exp\big{\{}\lambda \log |z|+ i \: \lambda \:\textup{arg}(z) \big{\}} \hspace{3ex} \hspace{3ex} \mbox{\it where} \hspace{3ex}  0< \textup{arg}(z) < \pi  
\end{equation*}
and use the notation of Proposition \ref{Proposition_Distribution_P(x)-ivarepsilon} and \cite[Chapter III, Section 2.4]{Gelfand_Shilov}. 
\begin{proposition}\label{Appendix_proposition_1}
Let $\psi \in \mathcal{S}(\mathbb{R}^{2n})$, then one has: 
\begin{equation}\label{GL_comparison_Distribution_GS}
\frac{1}{P^{n-1}} \big{[} \psi \big{]} =\lim_{\varepsilon \rightarrow 0} \int_{\mathbb{R}^{2n}} \psi(x) \big{(} P(x) +i \varepsilon \big{)}^{-n+1} \: \dd x= \Big{(} \big{(}P+i0\big{)}^{-n+1},\psi\Big{)}. 
\end{equation}
\end{proposition}
\begin{proof}
Let $\mathcal{L}$ be the ultra-hyperbolic operator in (\ref{Definition_classical_UHO_introduction}). For each $\varepsilon >0$ the following identity can be verified by a straightforward calculation:  
\begin{equation}\label{Relation_uh_operator_distribution}
\mathcal{L}\big{[}P(x)+i \varepsilon \big{]}^{\lambda+1}= 4(\lambda+1)(n+\lambda) \big{[}P(x)+i \varepsilon \big{]}^{\lambda} -i \varepsilon 4 \lambda(\lambda+1) \big{[} P(x)+i \varepsilon \big{]}^{\lambda-1}. 
\end{equation}
\par 
Assume in addition that $1 > \textup{Re}(\lambda) > -n$. Then it holds: 

\begin{equation}\label{appendix_limit_epsilon_zero}
\lim_{\varepsilon \rightarrow 0} \varepsilon \int_{\mathbb{R}^{2n}} \psi(x) \big{[} P(x) + i \varepsilon \big{]}^{\lambda-1} \; \dd x =0 \hspace{3ex} \mbox{\it for all } \hspace{3ex} \psi \in \mathcal{S}(\mathbb{R}^{2n}). 
\end{equation} 
In order to prove (\ref{appendix_limit_epsilon_zero})  we can use the integral representation (\ref{Gamma_function_integral_expression}) of the Gamma function and the identity (\ref{Chapter_7_Fourier_transform_integral}). By a decomposition of the integral similar to the calculation following Lemma \ref{Lemma_second_form_fundamental_solution_limit} 
one only needs to verify the estimate: 
\begin{align*}
0 
&\leq  \frac{1}{|\Gamma(1- \lambda)|} \Big{|} \int_1^{\infty} t^{-\lambda} e^{-t \varepsilon} \int_{\mathbb{R}^{2n}} \psi(x) e^{-iP(x)t} \: \dd x \dd t \Big{|}\\
& \leq \frac{1}{2^n |\Gamma(\lambda-1)|} \int_1^{\infty} t^{- \textup{Re}(\lambda)-n} e^{-t \varepsilon} \: \dd t \cdot \| \mathcal{F}^{-1} \psi \big{\|}_{L^1(\mathbb{R}^{2n})}
\end{align*}
and note that under the above assumption $\textup{Re}(\lambda)+n>0$: 
\begin{multline*}
0 \leq \varepsilon \int_1^{\infty} t^{- \textup{Re} (\lambda) -n} e^{-t \varepsilon} \: \dd t=\varepsilon^{\textup{Re} (\lambda) +n} \int_{\varepsilon}^{\infty}s^{-\textup{Re}(\lambda) -n} e^{-s} \: \dd s\\
\leq 
\begin{cases}
\frac{\varepsilon^{\textup{Re}(\lambda)+n}- \varepsilon}{-\textup{Re} (\lambda) -n+1}+ \varepsilon^{\textup{Re}(\lambda)+n}, &\textup{\it if}\hspace{2ex} \; \textup{Re} (\lambda) \ne -n+1, \\ 
\varepsilon \log\big{(} \frac{1}{\varepsilon} \big{)}+ \varepsilon^{\textup{Re}(\lambda)+n}, &\textup{\it if} \hspace{2ex} \; \textup{Re} (\lambda) = -n+1
\end{cases}
\longrightarrow 0, \hspace{2ex} \mbox{\it as} \hspace{2ex} \varepsilon \downarrow 0. 
\end{multline*}
Obviously, (\ref{appendix_limit_epsilon_zero}) remains valid in the case of $\textup{Re}(\lambda) \geq 1$. 
Hence, multiplying both sides of (\ref{Relation_uh_operator_distribution}) with $\psi \in \mathcal{S}(\mathbb{R}^{2n})$ and performing a partial integration over $\mathbb{R}^{2n}$ shows for all 
$\lambda \in \mathbb{C}$ with $\textup{Re}(\lambda)\geq -n$ (provided that the limit exists): 
\begin{equation*}
\lim_{\varepsilon \rightarrow 0} \int_{\mathbb{R}^{2n}} \big{[} \mathcal{L} \psi\big{]}(x) \big{(} P(x) +i \varepsilon \big{)}^{\lambda+1} \: \dd x
=4(\lambda+1)(n+ \lambda) \lim_{\varepsilon \rightarrow 0} \int_{\mathbb{R}^{2n}} \psi(x) \big{(} P(x) +i \varepsilon \big{)}^{\lambda} \: \dd x. 
\end{equation*}
Note that after a $k$-fold ($k \in \mathbb{N}$) iteration of the last equation one has:
\begin{equation}\label{Appendix_GL_k_fold_iteration}
\lim_{\varepsilon \rightarrow 0} \int_{\mathbb{R}^{2n}} 
\big{[} \mathcal{L}^k \psi \big{]}(x) \big{(} P(x) + i \varepsilon \big{)}^{\lambda+k} \; \dd x=\frac{1}{\Lambda(\lambda, k)}\lim_{\varepsilon \rightarrow 0} 
\int_{\mathbb{R}^{2n}} \psi(x) \big{(} P(x) +i \varepsilon \big{)}^{\lambda} \: \dd x,
\end{equation}
where we define: 
\begin{equation*}
\Lambda(\lambda, k):=\frac{1}{4^k \prod_{j=1}^k(\lambda+j)(n+\lambda+j-1)}. 
\end{equation*}
For sufficiently large integer $k$ and complex exponents $\lambda$ with $\textup{Re}(\lambda)\geq -n$ we show the existence of the limit above. 
If $\mu \in \mathbb{C}$ with $\textup{Re}(\mu)>0$ then:
\begin{align*}
\lim_{\varepsilon \rightarrow 0}
 \int_{\mathbb{R}^{2n}}  & \psi(x) \big{(} P(x) + i \varepsilon \big{)}^{\mu} \; \dd x= \\
=&
\lim_{\varepsilon \rightarrow 0} \int_{P>0} \psi(x) \exp \big{(} \mu \log |P(x)+i \varepsilon|+ \mu \: i\:  \textup{arg} (P(x)+i \varepsilon) \big{)} \: \dd x \\
&\hspace{3ex} + \lim_{\varepsilon \rightarrow 0} \int_{P<0} \psi(x) \exp \big{(} \mu \log |P(x)+i \varepsilon|+ \mu \:i\: \textup{arg} (P(x)+i \varepsilon) \big{)} \: \dd x\\
=& \int_{P>0} \psi(x) \exp \big{(} \mu \log P(x)\big{)} \dd x + e^{i\pi \mu} \int_{P<0} \psi(x) \exp \big{(} \mu \log (-P(x)) \big{)} \; \dd x\\
=& \big{(} P_+^{\mu}, \psi \big{)} + e^{i \pi \mu} \big{(} P_-^{\mu}, \psi \big{)}=\big{(} \big{(}P+i0\big{)}^{\mu}, \psi\big{)}. 
\end{align*}
Choose $k \in \mathbb{N}$ such that $\textup{Re} (\lambda)+k >0$. Inserting the above relation gives:
\begin{equation}\label{Appendix_meromorphic_extension}
\lim_{\varepsilon \rightarrow 0} \int_{\mathbb{R}^{2n}} \psi(x) \big{(} P(x) +i \varepsilon \big{)}^{\lambda} \: \dd x=\Lambda(\lambda, k)\Big{(} \big{(}P+i0\big{)}^{\lambda+k}, \mathcal{L}^k\psi\Big{)}
= \Big{(} \big{(}P+i0\big{)}^{\lambda},\psi\Big{)}. 
\end{equation}\par 
Hence the right hand side of  (\ref{Appendix_meromorphic_extension}) defines a meromorphic extension to the complex plane of the limit on the left with removable singularities at the negative integers 
$\lambda=-1, -2, \cdots,  -n+1$. Moreover, note that the left hand side of (\ref{Appendix_meromorphic_extension})  is continuous for $\lambda \in \{z \: : \: \textup{Re}(z) \geq -n+1 \}$. In fact, this can be 
seen from an integral representation of the limit based on the calculations following Lemma \ref{Lemma_second_form_fundamental_solution_limit}.  By choosing $\lambda=-n+1$ in (\ref{Appendix_meromorphic_extension}) the equality (\ref{GL_comparison_Distribution_GS}) follows. 
\end{proof}

\end{document}